\documentclass{article} %
\usepackage{arxiv,times}

\usepackage{amsmath,amsfonts,bm}

\def\eqref#1{(\ref{#1})}

\def\floor#1{\lfloor #1 \rfloor}
\def\1{\bm{1}}

\def\vx{{\bm{x}}}
\def\vy{{\bm{y}}}
\def\vz{{\bm{z}}}

\DeclareMathAlphabet{\mathsfit}{\encodingdefault}{\sfdefault}{m}{sl}
\SetMathAlphabet{\mathsfit}{bold}{\encodingdefault}{\sfdefault}{bx}{n}

\DeclareMathOperator*{\argmin}{arg\,min}

\usepackage{hyperref}       %
\usepackage{url}            %
\usepackage{booktabs}       %
\usepackage{amsfonts}       %
\usepackage{nicefrac}       %
\usepackage{microtype}      %
\usepackage{xcolor}         %
\usepackage{amsthm, amssymb, amsfonts, latexsym, mathtools}
\usepackage{algorithm}
\usepackage{algpseudocode}
\usepackage{comment}
\usepackage{here}
\usepackage{caption}
\usepackage{subcaption}
\usepackage{wrapfig}
\usepackage{multirow}
\usepackage{color}
\usepackage{wrapfig}
\usepackage[normalem]{ulem}
\usepackage{multirow}
\usepackage{thmtools}

\newtheorem{theorem}{Theorem}
\newtheorem{lemma}{Lemma}
\newtheorem{proposition}{Proposition}

\newtheorem{assumption}{Assumption}

\usepackage{cleveref}

\crefname{section}{Sec.}{Secs.}
\Crefname{section}{Section}{Sections}
\crefname{subsection}{Sec.}{Secs.}
\Crefname{subsection}{Section}{Sections}
\crefname{subsubsection}{Sec.}{Secs.}
\Crefname{subsubsection}{Section}{Sections}
\crefname{appendix}{App.}{Apps.}
\Crefname{appendix}{Appendix}{Appendices}

\crefname{figure}{Fig.}{Figs.}
\Crefname{figure}{Figure}{Figures}
\crefname{table}{Table}{Tables}
\Crefname{table}{Table}{Tables}

\crefname{equation}{Eq.}{Eqs.}
\Crefname{equation}{Equation}{Equations}

\crefname{theorem}{Theorem}{Theorems}
\Crefname{theorem}{Theorem}{Theorems}

\crefname{lemma}{Lemma}{Lemmas}
\Crefname{lemma}{Lemma}{Lemmas}

\crefname{proposition}{Proposition}{Propositions}
\Crefname{proposition}{Proposition}{Propositions}

\crefname{corollary}{Corollary}{Corollaries}
\Crefname{corollary}{Corollary}{Corollaries}

\crefname{assumption}{Assumption}{Assumptions}
\Crefname{assumption}{Assumption}{Assumptions}

\crefname{definition}{Def.}{Defs.}
\Crefname{definition}{Definition}{Definitions}

\crefname{remark}{Remark}{Remarks}
\Crefname{remark}{Remark}{Remarks}

\crefname{example}{Ex.}{Exs.}
\Crefname{example}{Example}{Examples}

\usepackage[colorinlistoftodos,bordercolor=orange,backgroundcolor=orange!20,linecolor=orange,textsize=scriptsize]{todonotes}

\title{Last-Iterate Convergence Rate of Normalized Gradient Descent under H\"older Smoothness}

\iclrfinalcopy

\author{Yuki Takezawa \\
Toyota Motor Corporation \\
\nolinkurl{yuki\_takezawa@mail.toyota.co.jp}
\And
Eduard Gorbunov \\
MBZUAI \\
\nolinkurl{eduard.gorbunov@mbzuai.ac.ae}
}

\begin{document}

\maketitle

\begin{abstract}

Normalized gradient descent is a widely studied adaptive optimization method. Most existing analyses focus on the best iterate or a weighted average of the iterates, whereas practical implementations typically return the last iterate. In this paper, we study the last-iterate convergence of normalized gradient descent for convex, $(\nu,M_\nu)$-H\"older-smooth objectives. For a constant stepsize, we establish an upper bound of $\mathcal{O}\bigl((\log^2(T)/T)^{(1+\nu)/2}\bigr)$, which contains a logarithmic overhead relative to the known $\mathcal{O}\bigl(T^{-(1+\nu)/2}\bigr)$ guarantees for the best and weighted-average iterates. For $\nu = 0$, this overhead is known to be unavoidable. We complement this analysis with numerical results based on the performance estimation problem (PEP), investigating the finite-horizon worst-case behavior in the smooth setting and whether the logarithmic overhead reflects an intrinsic limitation of constant-step normalized gradient descent. We then show that a linearly decreasing stepsize yields a last-iterate guarantee of $\mathcal{O}\bigl(T^{-(1+\nu)/2}\bigr)$, matching the order of the best-iterate/weighted-average guarantees without requiring knowledge of $\nu$ and $M_\nu$.

\end{abstract}

\section{Introduction}
\label{sec:introduction}

Neural networks have achieved remarkable empirical success across a wide range of applications, including computer vision \citep{russakovsky2015imagenet,radford2021learning,liu2021swin} and natural language processing \citep{bert2019devlin,brown2020language,deepseekai2024deepseekv2}.
Training such networks is challenging: their loss landscapes are often ill-conditioned, with curvature that varies widely across parameters and during training, which makes the performance of optimization methods highly sensitive to the choice of stepsize.
Adaptive optimization methods, which scale their updates using observed gradients, are designed to mitigate this sensitivity.
Building on AdaGrad \citep{duchi2011adagrad}, Adam and its variant AdamW \citep{kingma2015adam,loshchilov2018decoupled} have become the de facto standard for training deep neural networks, and Muon \citep{jordan2024muon,liu2025muon} has recently emerged as a strong alternative.

One of the most basic adaptive optimization methods is normalized gradient descent.
Despite its simple form, it is closely related to modern optimizers: without momentum (and with $\epsilon = 0$), Adam reduces to sign descent and Muon to orthogonalized gradient descent, and all three methods are instances of normalized steepest descent with respect to the $\ell_\infty$, spectral, and Euclidean norms, respectively \citep{bernstein2024old}.
Adam and Muon are among the most widely used optimizers for training large-scale models \citep{brown2020language,deepseekai2024deepseekv2,liu2025muon}.
Moreover, normalized gradient descent is universal: the same stepsize choice yields appropriate convergence rates across a wide range of problem classes, including smooth \citep{mishkin2024directional}, non-smooth \citep{nesterov2018lectures}, H\"older-smooth \citep{grimmer2019convergence,orabona2023normalized}, and $(L_0, L_1)$-smooth objectives \citep{vankov2025optimizing}.
However, most existing analyses bound the \textit{best} iterate or a \textit{weighted average} of the iterates, whereas practical implementations typically return the \textit{last} iterate.
This mismatch has motivated a growing body of work on last-iterate convergence across a wide range of optimization methods \citep{shamir2013stochastic,liu2024revisiting,zamani2024exact,zamani2025exact,defazio2024optimal}.
To the best of our knowledge, however, the last iterate of normalized gradient descent has so far been analyzed only by \citet{zamani2025exact}, and only for convex functions with bounded (sub)gradients, leaving open its behavior on problem classes that admit unbounded gradients.

In this paper, we analyze the last-iterate convergence rate of normalized gradient descent for convex\footnote{Although many practically important problems are nonconvex, understanding the convex regime remains important; see \cref{sec:importance_of_convexity} for a detailed discussion.} and H\"older-smooth objectives, a general smoothness class that includes both smooth and non-smooth functions as special cases.
We first show that, with a constant stepsize, the last iterate converges slower than the best and weighted-average iterates by a factor of order $\log^{1+\nu}(T)$, where $T$ is the number of iterations.
This overhead cannot be removed in general: for $\nu = 0$, our bound matches the lower bound of \citet{zamani2025exact} for the same constant stepsize, and for the smooth case ($\nu = 1$), numerical results based on the performance estimation problem (PEP) \citep{drori2014performance,taylor2017smooth} indicate that the last-iterate rate is indeed slower than $\mathcal{O}(1/T)$.
We then show that a linearly decreasing stepsize removes this overhead, yielding a last-iterate rate that matches the best-iterate and weighted-average guarantees.
\Cref{table:summary} compares our results with the existing ones.

Our contributions are summarized as follows.
\begin{itemize}
    \item \textbf{Last-iterate rate with a constant stepsize.} For convex and $(\nu, M_\nu)$-H\"older-smooth objectives, we prove that the last iterate of normalized gradient descent with a constant stepsize satisfies $f(\vx_T) - f(\vx^\star) = \mathcal{O}\bigl(M_\nu (\|\vx_1 - \vx^\star\|^2/T)^{(1+\nu)/2} (1 + \log T)^{1+\nu}\bigr)$ (\cref{theorem:convergenec_rate_of_last}). This rate is slower than the known best-iterate and weighted-average rates by a factor of order $\log^{1+\nu}(T)$. For $\nu = 0$, it recovers the result of \citet{zamani2025exact}, which is tight in this case.
    \item \textbf{Numerical evidence via PEP.} Using the performance estimation problem (PEP), we provide numerical evidence that, in the smooth case ($\nu = 1$), the last-iterate rate of normalized gradient descent with a constant stepsize is slower than $\mathcal{O}(1/T)$. This suggests that the logarithmic overhead is not an artifact of our analysis.
    \item \textbf{Removing the overhead with a decreasing stepsize.} We show that a linearly decreasing stepsize yields the last-iterate rate $\mathcal{O}\bigl(M_\nu (\|\vx_1 - \vx^\star\|^2/T)^{(1+\nu)/2}\bigr)$ (\cref{theorem:convergenec_rate_of_last_with_adaptive_stepsize}), matching the best-iterate and weighted-average guarantees. This stepsize does not require knowledge of $\nu$ or $M_\nu$, and for $\nu = 0$ the result recovers that of \citet{zamani2025exact}.
\end{itemize}

\paragraph{Notation.}
For $\vx, \vy \in \mathbb{R}^d$, $\langle \vx, \vy \rangle$ denotes the standard inner product and $\|\vx\| \coloneqq \sqrt{\langle \vx, \vx \rangle}$ the Euclidean norm.
$T$ denotes the number of iterations, and $\log$ denotes the natural logarithm. In formal statements (theorems, lemmas, and propositions), where the dependence on all problem parameters is written explicitly, $\mathcal{O}(\cdot)$ hides only absolute numerical constants, i.e., constants independent of $T$, $\nu$, $M_\nu$, and $\|\vx_1 - \vx^\star\|$.
In informal discussions, such as $\mathcal{O}(1/T)$ or $\mathcal{O}(T^{-\frac{1+\nu}{2}})$, we use $\mathcal{O}(\cdot)$ to indicate the dependence on $T$ only, treating the remaining parameters as constants.

\begin{table}[t]
\caption{Summary of the convergence rates for normalized gradient descent assuming that $f$ is convex and $(\nu, M_\nu)$-H\"older smooth. Numerical constants are omitted in the rates.}
\label{table:summary}
\centering
\resizebox{\textwidth}{!}{%
\begin{tabular}{llcccc}
\toprule
 & Reference & Convergence Criteria & $\nu$ & Stepsize & Rate \\
\midrule
\multirow{5}{*}{\parbox{1.5cm}{\centering constant\\stepsize}}
& \citet{grimmer2019convergence} & $\min_t f (\vx_t) - f (\vx^\star)$ & $\nu\in[0,1]$ & $\tfrac{\| \vx_1 - \vx^\star\|}{\sqrt{T}}$ & $M_\nu \left( \tfrac{\| \vx_1 - \vx^\star\|}{\sqrt{T}} \right)^{1 + \nu}$ \\
& \citet{orabona2023normalized} & $f \left( \tfrac{1}{\sum_{t=1}^T \frac{1}{\| \nabla f (\vx_t) \|}} \sum_{t=1}^T \tfrac{\vx_t}{\| \nabla f (\vx_t) \|} \right) - f (\vx^\star)$ & $\nu \in [0,1]$ & $\tfrac{\| \vx_1 - \vx^\star\|}{\sqrt{T}}$ & $M_\nu \left( \tfrac{\| \vx_1 - \vx^\star\|}{\sqrt{T}} \right)^{1 + \nu}$ \\
& \citet{zamani2025exact} & $f (\vx_T) - f (\vx^\star)$ & $\nu=0$ & $\tfrac{\| \vx_1 - \vx^\star\|}{\sqrt{T}}$ & $\tfrac{M_\nu \| \vx_1 - \vx^\star\|}{\sqrt{T}} (1 + \log (T))$ \\
& \Cref{theorem:convergenec_rate_of_last} \textbf{(new)}& $f (\vx_T) - f (\vx^\star)$ & $\nu\in[0,1]$ & $\tfrac{\| \vx_1 - \vx^\star\|}{\sqrt{T}}$ & $M_\nu \left( \tfrac{\| \vx_1 - \vx^\star\|}{\sqrt{T}} \right)^{1 + \nu} (1 + \log (T))^{1 + \nu}$ \\
\midrule
\multirow{3}{*}{\parbox{1.8cm}{\centering linearly\\decreasing\\stepsize}}
& \citet{zamani2025exact} & $f (\vx_T) - f (\vx^\star)$ & $\nu=0$ & $\tfrac{\| \vx_1 - \vx^\star\|}{\sqrt{T}} \left( 1 - \tfrac{t}{T} \right)$ & $\tfrac{M_\nu \| \vx_1 - \vx^\star\|}{\sqrt{T}}$ \\
& \Cref{theorem:convergenec_rate_of_last_with_adaptive_stepsize} \textbf{(new)} & $f (\vx_T) - f (\vx^\star)$ & $\nu\in[0,1]$ & $\tfrac{\| \vx_1 - \vx^\star\|}{\sqrt{T}} \left( 1 - \tfrac{t}{T+1} \right)$ & $M_\nu \left( \tfrac{\| \vx_1 - \vx^\star\|}{\sqrt{T}} \right)^{1 + \nu}$ \\
\bottomrule
\end{tabular}}

\vskip - 0.2 in
\end{table}

\section{Related Work}
\label{sec:related_work}

\paragraph{Convergence Analysis of Normalized Gradient Descent.}
Normalized gradient descent and its variants are among the simplest adaptive optimization methods and have been studied in a variety of settings \citep{nesterov2018lectures,grimmer2019convergence,cutkosky2020momentun,orabona2023normalized,mishkin2024directional,vankov2025optimizing}.
Here, we focus on the convex, deterministic setting considered in this paper.
Most existing analyses concern the best iterate: \citet{nesterov2018lectures} and \citet{mishkin2024directional} cover both smooth and non-smooth objectives, \citet{grimmer2019convergence} considers H\"older-smooth objectives, and \citet{vankov2025optimizing} considers $(L_0, L_1)$-smooth objectives.
\citet{orabona2023normalized} analyzes a weighted average of the iterates instead.
However, practical implementations typically return the last iterate rather than the best or an averaged one.
To the best of our knowledge, \citet{zamani2025exact} is the only prior work on the last-iterate convergence of normalized gradient descent.
For convex functions with bounded subgradients, they showed that, with the constant stepsize $\|\vx_1 - \vx^\star\|/\sqrt{T}$, the last iterate converges at the rate $\mathcal{O}(\log T / \sqrt{T})$, which is slower than the $\mathcal{O}(1/\sqrt{T})$ best-iterate rate, and they established a matching lower bound for this stepsize; moreover, they showed that even the optimally tuned constant stepsize only reduces the overhead to $\sqrt{\log T}$, so that a logarithmic loss cannot be avoided with any constant stepsize. 
They further showed that a linearly decreasing stepsize improves the last-iterate rate to $\mathcal{O}(1/\sqrt{T})$.
However, the bounded-subgradient assumption is restrictive: it excludes, for example, smooth objectives such as quadratics, whose gradients are unbounded on $\mathbb{R}^d$.
In contrast, we analyze the last iterate under H\"older smoothness, which covers both smooth and non-smooth objectives, and our results recover those of \citet{zamani2025exact} as the special case $\nu = 0$.

\paragraph{Last-Iterate Convergence of Optimization Methods.}
Convergence analyses of optimization methods have traditionally bounded the best or an averaged iterate, whereas practical implementations typically return the last iterate.
Motivated by this gap between theory and practice, a growing body of work analyzes the last iterate directly, including stochastic gradient descent \citep{shamir2013stochastic,liu2024revisiting}, subgradient methods \citep{zamani2025exact} and their variant with the Polyak stepsize \citep{zamani2024exact}, shuffling gradient methods \citep{liu2024on}, momentum methods \citep{li2022on}, and AdaGrad-Norm \citep{preobrazhenskaia2026last}.
A recurring finding is that, under standard stepsize choices, the last iterate can converge more slowly than the best or averaged iterate, typically by a logarithmic factor \citep{shamir2013stochastic,zamani2025exact}, and this loss is unavoidable for several standard stepsize choices \citep{harvey2019tight,kornowski2026gradient}, and that suitably decaying stepsize schedules can close this gap \citep{jain2019making,zamani2025exact,defazio2024optimal}.
In this direction, \citet{defazio2024optimal} proposed a general framework that converts regret bounds of online learning algorithms into last-iterate guarantees on $f(\vx_T) - f(\vx^\star)$, which in particular justifies linear-decay schedules. However, this framework does not directly apply to normalized gradient descent: its regret bound holds for the normalized gradients $\nabla f(\vx_t)/\|\nabla f(\vx_t)\|$ rather than the gradients themselves, and under H\"older smoothness the gradient norms are not bounded away from zero.

\section{Preliminaries}
In this section, we introduce the problem setting and assumptions, and then summarize the existing guarantees for normalized gradient descent. 
We consider the unconstrained minimization problem
\begin{align}
\label{eq:problem}
    \min_{\vx \in \mathbb{R}^d} f (\vx),
\end{align}
and assume that it admits a minimizer $\vx^\star$. We study normalized gradient descent:
\begin{align}
    \vx_{t+1} = \vx_t - \tilde{\eta}_t \frac{\nabla f (\vx_t)}{\| \nabla f (\vx_t)\|},
\end{align}
where $\tilde{\eta}_t > 0$ is a possibly time-varying stepsize.
Since $\| \vx_{t+1} - \vx_t \| = \tilde{\eta}_t$, the stepsize $\tilde{\eta}_t$ determines the length of each step regardless of the gradient magnitude, whereas $\eta_t$ denotes the resulting effective stepsize applied to $\nabla f(\vx_t)$.
If $\nabla f(\vx_t) = 0$ for some $t$, then $\vx_t$ is a minimizer by convexity; hence, without loss of generality, we assume that $\nabla f(\vx_t) \neq 0$ for all $t$.

Throughout this paper, we assume that the objective $f$ is convex and H\"older smooth, which are standard assumptions in the optimization literature \citep{nesterov2016universal}.
\begin{assumption}[Convexity]
\label{assumption:convex}
For any $\vx, \vy \in \mathbb{R}^d$, it holds that
\begin{align*}
    f (\vy) \geq f (\vx) + \langle \nabla f (\vx), \vy - \vx \rangle.
\end{align*}
\end{assumption}

\begin{assumption}[($\nu, M_\nu$)-H\"older Smoothness]
\label{assumption:holder_smoothness}
There exist $\nu \in [0, 1]$ and $M_\nu > 0$ such that, for any $\vx, \vy \in \mathbb{R}^d$,
\begin{align}
\label{eq:holder_smooth}
    \| \nabla f (\vx) - \nabla f (\vy) \| \leq M_\nu \| \vx - \vy \|^\nu.
\end{align}
\end{assumption}
H\"older smoothness interpolates between non-smooth and smooth objectives.
For $\nu = 1$, \cref{eq:holder_smooth} is the standard $M_1$-smoothness assumption, i.e., Lipschitz continuity of the gradient, which allows the gradient norm to grow without bound, as for quadratic functions.
For $\nu = 0$, \cref{eq:holder_smooth} is essentially the bounded-gradient assumption: it holds with $M_0 = 2G$ whenever $\| \nabla f(\vx) \| \leq G$ for all $\vx$, and, conversely, taking $\vy = \vx^\star$ gives $\| \nabla f(\vx) \| \leq M_0$.
In this case, $f$ need not be differentiable, and $\nabla f(\vx)$ denotes an arbitrary subgradient of $f$ at $\vx$, with $\nabla f(\vx^\star) = 0$.
For $\nu = 0$, our proofs use only the subgradient inequality of \cref{assumption:convex}, the bound $\|\nabla f(\vx)\| \leq M_0$, and the upper bound $f(\vy) \leq f(\vx) + \langle \nabla f(\vx), \vy - \vx \rangle + M_0 \|\vy - \vx\|$, all of which hold for subgradients of convex functions with $M_0/2$-bounded subgradients; hence our results for $\nu = 0$ apply to the normalized subgradient method on non-smooth convex functions, which is the setting of \citet{zamani2025exact}.

\subsection{Existing Analysis of Normalized Gradient Descent}
As discussed in \cref{sec:related_work}, most existing analyses of normalized gradient descent bound either the best iterate or a weighted average of the iterates.
Under H\"older smoothness, the following two guarantees are known (see also \cref{table:summary}).

\begin{proposition}[\citet{grimmer2019convergence}]
\label{proposition:best_iterate}
Suppose that \Cref{assumption:convex,assumption:holder_smoothness} hold. Then, for any $T \geq 1$, normalized gradient descent with $\tilde{\eta}_t = \frac{\| \vx_1 - \vx^\star\|}{\sqrt{T}}$ satisfies:
\begin{align}
    \min_{t \in \{1,2,\dots,T\}} f (\vx_t) - f (\vx^\star) \leq \mathcal{O} \left( M_\nu \left( \frac{\| \vx_1 - \vx^\star\|^2}{T} \right)^\frac{1 + \nu}{2} \right).
\end{align}
\end{proposition}

\begin{proposition}[\citet{orabona2023normalized}]
\label{proposition:weight_averaged_iterate}
Suppose that \Cref{assumption:convex,assumption:holder_smoothness} hold. Then, for any $T \geq 1$, normalized gradient descent with $\tilde{\eta}_t = \frac{\| \vx_1 - \vx^\star\|}{\sqrt{T}}$ satisfies:
\begin{align}
    f (\tilde{\vx}_T) - f (\vx^\star) \leq \mathcal{O} \left( M_\nu \left( \frac{\| \vx_1 - \vx^\star\|^2}{T} \right)^\frac{1 + \nu}{2} \right),
\end{align}
where $\tilde{\vx}_T = \frac{1}{\sum_{t=1}^T \frac{1}{\| \nabla f (\vx_t) \|}} \sum_{t=1}^T \frac{\vx_t}{\| \nabla f (\vx_t) \|}$.
\end{proposition}

{\Cref{proposition:best_iterate} bounds the best iterate, i.e., the iterate with the smallest function value, which can be identified only by evaluating $f$ at every iterate.
\Cref{proposition:weight_averaged_iterate} bounds a weighted average of the iterates, with weights proportional to the effective stepsizes $\eta_t$.
In practice, however, normalized gradient descent typically returns the last iterate $\vx_T$, which neither result covers.
To the best of our knowledge, the only existing last-iterate guarantee for normalized gradient descent is that of \citet{zamani2025exact}, which is restricted to $\nu = 0$.
In the next section, we establish last-iterate guarantees for the whole H\"older family $\nu \in [0, 1]$.

\section{Normalized Gradient Descent with a Constant Stepsize}
\label{sec:constant_stepsize}

In this section, we consider normalized gradient descent with a constant stepsize, i.e., $\tilde{\eta}_t = \tilde{\eta}$ for all $t$.
All proofs are deferred to \cref{sec:proof_of_last}.

\begin{theorem}
\label{theorem:convergenec_rate_of_last}
Suppose that Assumptions~\ref{assumption:convex} and \ref{assumption:holder_smoothness} hold.
Then, for any $T \geq 1$, normalized gradient descent with $\tilde{\eta}_t = \frac{\| \vx_1 - \vx^\star\|}{\sqrt{T}}$ satisfies
\begin{align*}
    f (\vx_T) - f (\vx^\star) \leq \mathcal{O} \left( M_\nu \left( \frac{\| \vx_1 - \vx^\star\|^2}{T} \right)^\frac{1 + \nu}{2} (1 + \log (T))^{1 + \nu} \right).
\end{align*}
\end{theorem}

\paragraph{Discussion.} 
\Cref{theorem:convergenec_rate_of_last} provides a last-iterate guarantee for normalized gradient descent under H\"older smoothness.
As discussed in \cref{sec:related_work}, \citet{zamani2025exact} analyzed the last iterate only under the bounded-gradient assumption, i.e., for $\nu = 0$. \Cref{theorem:convergenec_rate_of_last} recovers their result up to absolute constants as a special case and extends it to the whole H\"older family $\nu \in [0,1]$, which, in particular, includes smooth objectives with unbounded gradients. See the detailed comparison in \cref{table:summary}.

Compared with the best-iterate and weighted-average guarantees in \cref{proposition:best_iterate,proposition:weight_averaged_iterate}, the bound in \cref{theorem:convergenec_rate_of_last} is worse by a factor of $(1 + \log(T))^{1 + \nu}$.
For $\nu = 0$, this overhead is not an artifact of the analysis: \citet{zamani2025exact} established a lower bound for normalized gradient descent with the constant stepsize $\|\vx_1 - \vx^\star\|/\sqrt{T}$ that matches \cref{theorem:convergenec_rate_of_last} in this case, so the logarithmic factor cannot be removed in general for this stepsize.
For $\nu > 0$, no such lower bound is currently known. In \cref{sec:numerical_evaluation}, we therefore study the smooth case $\nu = 1$ numerically and find that the last iterate of normalized gradient descent with a constant stepsize converges slower than $\mathcal{O}(1/T)$ in the worst case, over the range of horizons we can solve.

This behavior is particularly notable in the smooth case.
For convex and smooth objectives, gradient descent with a constant stepsize attains the rate $\mathcal{O}(1/T)$ for the last iterate \citep{nesterov2018lectures}, which is the same order as the best-iterate rate of normalized gradient descent in \cref{proposition:best_iterate}.
In contrast, the numerical results in \cref{sec:numerical_evaluation} indicate that normalized gradient descent does not attain this rate for the last iterate.
Thus, normalization can degrade the last-iterate behavior relative to gradient descent, even though it does not affect the best-iterate rate.

\subsection{Proof Sketch}
\label{sec:proof_sketch_of_constant_stepsize}

We now outline the proof of \cref{theorem:convergenec_rate_of_last}, explaining where the logarithmic factor comes from and why the case $\nu > 0$ requires an argument different from that of \citet{zamani2025exact}.

\paragraph{Key lemma.} Both of our main results rely on the following lemma, which is a variant of Lemma~2.1 of \citet{zamani2025exact} and of an argument in the proof of Lemma~4.1 of \citet{liu2024revisiting}.
Its proof compares each iterate $\vx_t$ with a reference point $\vz_t$ that is a convex combination of $\vx^\star$ and the previous iterates, with weights determined by the sequence $\{ v_t \}_t$.

\begin{restatable}{lemma}{zamaniLemma}
\label{lemma:zamani}
Suppose that \cref{assumption:convex} holds.
Then, for any sequence $\{ v_t \}_{t=0}^T$ such that $0 < v_0 \leq v_1 \leq \dots \leq v_T$, normalized gradient descent satisfies
\begin{align*}
    2 \sum_{t=1}^T \left[ \eta_t v_t - \left( \sum_{s=t}^T \eta_s \right) (v_t - v_{t-1}) \right] \left( f (\vx_t) - f (\vx^\star) \right)
    &\leq v_{0} \| \vx_1 - \vx^\star \|^2 + \sum_{t=1}^T v_t \tilde{\eta}_t^2.
\end{align*}
\end{restatable}

\paragraph{Reduction to two terms.}
To simplify the notation, we write $\eta_t \coloneqq \frac{\tilde{\eta}_t}{\| \nabla f (\vx_t) \|}$.
To isolate the last iterate, we choose
\begin{align}
\label{eq:definition_of_v_for_last_iterate}
    v_T \coloneqq \frac{1}{\eta_T},
    \quad v_t \coloneqq \frac{1}{\sum_{s=t+1}^T \eta_s} \quad \text{for all $t \in \{0, 1, \dots, T-1\}$}.
\end{align}
This sequence is nondecreasing, and with this choice the coefficient of $f(\vx_t) - f(\vx^\star)$ on the left-hand side of \cref{lemma:zamani} vanishes for every $t < T$ and equals $2$ for $t = T$.
Hence, for $\tilde{\eta}_t \coloneqq \tilde{\eta}$, we obtain the following by \cref{lemma:zamani}:
\begin{align}
\label{eq:zamani_last}
    2 \left( f (\vx_T) - f (\vx^\star) \right)
    &\leq v_{0} \| \vx_1 - \vx^\star \|^2 + \tilde{\eta}^2 \sum_{t=1}^T v_t.
\end{align}
Hence, it remains to bound $v_0$ and $\tilde{\eta}^2 \sum_{t=1}^T v_t$ from above, and both quantities are controlled by the gradient norms at the iterates.
For $v_0$, dropping the term $s = T$ and applying \cref{lemma:bound_of_grad_norm}, a consequence of \cref{assumption:holder_smoothness} that bounds $\| \nabla f (\vx) \|$ by $2 M_\nu^{\frac{1}{1+\nu}} (f(\vx) - f(\vx^\star))^{\frac{\nu}{1+\nu}}$, we get
\begin{align}
\label{eq:upper_bound_of_v0}
    v_0 &= \frac{1}{\tilde{\eta} \sum_{s=1}^T \frac{1}{\| \nabla f (\vx_s) \|}}
    \leq \frac{1}{\tilde{\eta} \sum_{s=1}^{T-1} \frac{1}{\| \nabla f (\vx_s) \|}}
    \leq \frac{2 M_\nu^\frac{1}{1 + \nu}}{\tilde{\eta}} \left( \sum_{s=1}^{T-1}  \left( f (\vx_s) - f (\vx^\star) \right)^{- \frac{\nu}{1 + \nu}} \right)^{-1}.
\end{align}
Similarly, using the inequality between the harmonic and the arithmetic mean, \cref{lemma:bound_of_grad_norm}, the definition of $v_T$, and a change of the order of summation, we get
\begin{align}
\label{eq:upper_bound_of_sum_of_v}
    \tilde{\eta}^2 \sum_{t=1}^T v_t
    &= \tilde{\eta}^2  v_T + \tilde{\eta} \sum_{t=1}^{T-1} \frac{1}{\sum_{s=t+1}^T \frac{1}{\| \nabla f (\vx_s) \|}} \nonumber \\
    &\leq \tilde{\eta}^2  v_T + \tilde{\eta} \sum_{t=1}^{T-1} \frac{1}{(T - t)^2} \sum_{s=t+1}^T \| \nabla f (\vx_s) \| \\
    &\leq \tilde{\eta}^2 v_T + 2 M_\nu^\frac{1}{1 + \nu} \tilde{\eta} \sum_{t=1}^{T-1} \frac{1}{(T - t)^2} \sum_{s=t+1}^T  \left( f (\vx_s) - f (\vx^\star) \right)^\frac{\nu}{1 + \nu} \nonumber \\
    &\leq 6 M_\nu^\frac{1}{1 + \nu} \tilde{\eta} \left( f (\vx_T) - f (\vx^\star) \right)^\frac{\nu}{1 + \nu}
    + 2 M_\nu^\frac{1}{1 + \nu} \tilde{\eta}  \sum_{s=2}^{T-1}  \left( f (\vx_s) - f (\vx^\star) \right)^\frac{\nu}{1 + \nu} \sum_{t=1}^{s-1} \frac{1}{(T - t)^2} \nonumber.
\end{align}
The details are given in \cref{sec:proof_of_last}.
When $\nu = 0$, the factors $(f(\vx_s) - f(\vx^\star))^{\frac{\nu}{1 + \nu}}$ are equal to one, and the last-iterate rate follows directly from \cref{eq:zamani_last,eq:upper_bound_of_v0,eq:upper_bound_of_sum_of_v}, as in \citet{zamani2025exact}.
For $\nu \in (0, 1]$, however, the bounds on $v_0$ and $\tilde{\eta}^2 \sum_{t=1}^T v_t$ depend on the suboptimality gaps $f(\vx_s) - f(\vx^\star)$ of the earlier iterates, i.e., on the very quantities we want to bound.
This is what makes the case $\nu > 0$ harder: we resolve it by proving the following bound for \emph{every} iterate by induction on $t$.

\begin{restatable}{lemma}{LemmaaveragedRateWithConstant}
\label{lemma:average_rate_with_constant}
Suppose that Assumptions~\ref{assumption:convex} and \ref{assumption:holder_smoothness} hold, and let $\tilde{\eta}_t = \tilde{\eta} > 0$ for all $t$.
Then, for every $t \geq 1$ such that $\tilde{\eta} \leq \frac{\| \vx_1 - \vx^\star\|}{\sqrt{t}}$, normalized gradient descent satisfies
\begin{align}
\label{eq:last_iterate_for_every_t}
    \!\! f (\vx_t) - f (\vx^\star) \leq M_\nu \left( 1200 \left( \frac{\| \vx_1 - \vx^\star \|^{2}}{\tilde{\eta} t} \right)^{1 + \nu}
    \!\!\!\!\!\!\!\!\! + 200 \left( \frac{\| \vx_1 - \vx^\star\|^2}{t} \right)^\frac{1 + \nu}{2} \right) (1 + \log (t))^{1 + \nu}.
\end{align}
\end{restatable}
Since $\tilde{\eta} = \|\vx_1 - \vx^\star\|/\sqrt{T}$ satisfies $\tilde{\eta} \leq \|\vx_1 - \vx^\star\|/\sqrt{t}$ for every $t \leq T$, \cref{theorem:convergenec_rate_of_last} follows from \cref{lemma:average_rate_with_constant} with $t = T$, with the absolute constant $1400$.

\paragraph{Proof Sketch of \cref{lemma:average_rate_with_constant}.}
We sketch the induction step, which also shows where the logarithmic factor comes from; the base case $t = 1$ follows from \cref{lemma:holder_smooth}.

Assume that \cref{eq:last_iterate_for_every_t} holds for all $1 \leq t \leq T-1$; we show that it holds for $t = T$ whenever $\tilde{\eta} \leq \|\vx_1 - \vx^\star\|/\sqrt{T}$.
By \cref{eq:zamani_last}, it suffices to bound $v_0$ and $\tilde{\eta}^2 \sum_{t=1}^T v_t$.
Combining \cref{eq:upper_bound_of_v0} with the induction hypothesis (which applies to every $s \leq T-1$ because $\tilde{\eta} \leq \|\vx_1 - \vx^\star\|/\sqrt{s}$) and elementary manipulations, we obtain
\begin{align*}
    v_0 \leq 320 (1 + \log (T))^\nu M_\nu \left( \frac{\| \vx_1 - \vx^\star\|^{2\nu}}{\tilde{\eta}^{1 + \nu} T^{1 + \nu}} \right).
\end{align*}
The factor $(1 + \log(T))^\nu$ is inherited from the induction hypothesis through $(f(\vx_s) - f(\vx^\star))^{\frac{\nu}{1+\nu}}$; the bound on $v_0$ itself does not introduce any additional logarithmic factor.

The additional logarithmic factor arises in the bound on $\tilde{\eta}^2 \sum_{t=1}^T v_t$, so we discuss it in more detail.
By \cref{eq:upper_bound_of_sum_of_v}, it suffices to bound
\begin{align*}
    &\sum_{s=2}^{T-1}  \left( f (\vx_s) - f (\vx^\star) \right)^\frac{\nu}{1 + \nu} \sum_{t=1}^{s-1} \frac{1}{(T - t)^2} \\
    &= \underbrace{\sum_{s=2}^{\floor{\frac{T-1}{2}}}  \left( f (\vx_s) - f (\vx^\star) \right)^\frac{\nu}{1 + \nu} \sum_{t=1}^{s-1} \frac{1}{(T - t)^2}}_{\mathcal{T}_1}
    + \underbrace{\sum_{s=\floor{\frac{T-1}{2}} + 1}^{T-1}  \left( f (\vx_s) - f (\vx^\star) \right)^\frac{\nu}{1 + \nu} \sum_{t=1}^{s-1} \frac{1}{(T - t)^2}}_{\mathcal{T}_2}.
\end{align*}

In $\mathcal{T}_1$, the index $s$ is at most $\frac{T-1}{2}$, so $T - t \geq \frac{T}{2}$ for all $t \leq s - 1$ and the inner sum is at most $\frac{4s}{T^2}$; combining this with \cref{eq:last_iterate_for_every_t} for $s \leq T-1$, we obtain, without any additional logarithmic factor,
\begin{align*}
   \mathcal{T}_1
   &\leq 200 M_\nu^\frac{\nu}{1 + \nu} (1 + \log (T))^\nu \left( \frac{\| \vx_1 - \vx^\star \|^{2 \nu}}{\tilde{\eta}^{\nu} T^\nu } \right).
\end{align*}

In $\mathcal{T}_2$, the index $s$ can be close to $T$, and then the inner sum $\sum_{t=1}^{s-1} \tfrac{1}{(T-t)^2} \leq \tfrac{1}{T - s}$ is no longer small; summing $\tfrac{1}{T-s}$ over $s$ produces a harmonic sum, i.e., an additional logarithmic factor.
Specifically, using $s \geq \floor{\frac{T-1}{2}} + 1$, $T \geq 2$, and \cref{eq:last_iterate_for_every_t}, we have
\begin{align}
    \label{eq:t2}
    \mathcal{T}_2
    &\leq M_\nu^\frac{\nu}{1 + \nu} (1 + \log (T))^{\nu} \sum_{s=\floor{\frac{T-1}{2}}}^{T-1}  \left( 70 \left( \frac{\| \vx_1 - \vx^\star \|^{2}}{\tilde{\eta} T} \right)^{\nu}
    + 30 \left( \frac{\| \vx_1 - \vx^\star\|^2}{T} \right)^\frac{\nu}{2} \right) \sum_{t=1}^{s-1} \frac{1}{(T - t)^2} \nonumber \\
    &\leq  M_\nu^\frac{\nu}{1 + \nu} (1 + \log (T))^{\nu} \sum_{s=\floor{\frac{T-1}{2}}}^{T-1}  \left( 70 \left( \frac{\| \vx_1 - \vx^\star \|^{2}}{\tilde{\eta} T} \right)^{\nu}
    + 30 \left( \frac{\| \vx_1 - \vx^\star\|^2}{T} \right)^\frac{\nu}{2} \right) \frac{1}{T - s} \\
    &\leq M_\nu^\frac{\nu}{1 + \nu} (1 + \log (T))^{1 + \nu} \left( 70 \left( \frac{\| \vx_1 - \vx^\star \|^{2}}{\tilde{\eta} T} \right)^{\nu}
    + 30 \left( \frac{\| \vx_1 - \vx^\star\|^2}{T} \right)^\frac{\nu}{2} \right). \nonumber
\end{align}
The last step uses $\sum_{s=\floor{\frac{T-1}{2}}}^{T-1} \frac{1}{T-s} \leq 1 + \log(T)$.
This harmonic sum is the source of the additional logarithmic factor, and it is precisely the reason why the induction hypothesis \cref{eq:last_iterate_for_every_t} must itself contain a logarithmic factor.
Combining the bounds on $v_0$, $\mathcal{T}_1$, and $\mathcal{T}_2$ with \cref{eq:zamani_last,eq:upper_bound_of_sum_of_v} and using $\tilde{\eta} \leq \|\vx_1 - \vx^\star\|/\sqrt{T}$, we obtain
\begin{align*}
    f (\vx_T) - f (\vx^\star)
    &\leq \frac{4 M_\nu^\frac{1}{1 + \nu} \| \vx_1 - \vx^\star\|^2}{T \tilde{\eta}} \left( f (\vx_T) - f (\vx^\star) \right)^\frac{\nu}{1 + \nu} \\
    &+  M_\nu (1 + \log (T))^{1 + \nu} \left( 430 \left( \frac{\| \vx_1 - \vx^\star \|^{2}}{\tilde{\eta} T} \right)^{1 + \nu}
    + 30 \left( \frac{\| \vx_1 - \vx^\star\|^2}{T} \right)^\frac{1 + \nu}{2} \right),
\end{align*}
and \cref{lemma:self_improving_bound}, a consequence of Young's inequality, yields \cref{eq:last_iterate_for_every_t} for $t = T$.

\subsection{Numerical Evaluation}
\label{sec:numerical_evaluation}

In this section, we use the performance estimation problem (PEP) \citep{drori2014performance,taylor2017smooth} to examine numerically whether the logarithmic overhead in \cref{theorem:convergenec_rate_of_last} is an artifact of our analysis.
For $\nu = 0$, it is not: the lower bound of \citet{zamani2025exact} matches \cref{theorem:convergenec_rate_of_last} for the stepsize $\|\vx_1 - \vx^\star\|/\sqrt{T}$.
For $\nu > 0$, no lower bound is known, so we focus on the smooth case $\nu = 1$ and ask whether the last iterate attains the $\mathcal{O}(1/T)$ rate of the best iterate.

PEP computes the exact worst-case performance of a first-order method over a class of functions by solving a semidefinite program.
Since we are interested in the dependence on $T$, we set $M_\nu = 1$ and $\| \vx_1 - \vx^\star \| = 1$.
The worst-case last-iterate suboptimality of normalized gradient descent after $T$ iterations is then the optimal value of
\begin{align}
\begin{alignedat}{2}
    \max_{f} f(\vx_T) - f(\vx^\star)
    \quad \text{s.t.} \quad &  \|\vx_1 - \vx^\star\| = 1, \\
    & \vx_{t+1} = \vx_t - \frac{\|\vx_1 - \vx^\star\|}{\sqrt{T}} \frac{\nabla f(\vx_t)}{\|\nabla f(\vx_t)\|}, \\
    &\text{$f$ satisfies \cref{assumption:convex,assumption:holder_smoothness} with $M_\nu=1$ and $\nu=1$}.
\end{alignedat}
\label{eq:pep1}
\end{align}

Because the normalization $\nabla f(\vx_t) / \|\nabla f(\vx_t)\|$ is nonlinear in $\nabla f(\vx_t)$, \cref{eq:pep1} cannot be cast as a semidefinite program directly.
We therefore fix the gradient norms $\{ g_t \}_t$ in advance and solve
\begin{align}
\begin{alignedat}{2}
    \max_{f} f(\vx_T) - f(\vx^\star)
    \quad \text{s.t.} \quad &  \|\vx_1 - \vx^\star\| = 1, \\
    & \vx_{t+1} = \vx_t - \frac{\|\vx_1 - \vx^\star\|}{\sqrt{T}} \frac{\nabla f(\vx_t)}{g_t}, \\
    & \| \nabla f (\vx_t) \| = g_t \; \text{for all $1 \leq t \leq T-1$}, \\
    &\text{$f$ satisfies \cref{assumption:convex,assumption:holder_smoothness} with $M_\nu=1$ and $\nu=1$},
\end{alignedat}
\label{eq:pep2}
\end{align}
which can be formulated as a semidefinite program: the update is now linear in $\nabla f(\vx_t)$, and the norm constraints are linear in the Gram matrix of the gradients; see \citet{drori2014performance,taylor2017smooth} for the details.
We solve it with PEPit \citep{goujaud2024pepit}.
Every feasible point of \cref{eq:pep2} is feasible for \cref{eq:pep1}; hence, for any choice of $\{ g_t \}_t$, the optimal value of \cref{eq:pep2} is a lower bound on the worst-case suboptimality of normalized gradient descent after $T$ iterations, and maximizing it over $\{ g_t \}_t$ would recover \cref{eq:pep1}.
Since this maximization is intractable, and since a single sequence suffices to exhibit slow convergence, we tried several choices of $\{ g_t \}_t$ and report the results for $g_t = c / \sqrt{T}$ with $c \in \{0.1, 0.06, 0.02\}$, for which the observed decay is slower than $1/T$.
Other choices of $\{ g_t \}_t$ may lead to even larger worst-case values.

\begin{figure}[!t]
    \centering
    \vskip -0.1 in
    \includegraphics[height=4.8cm]{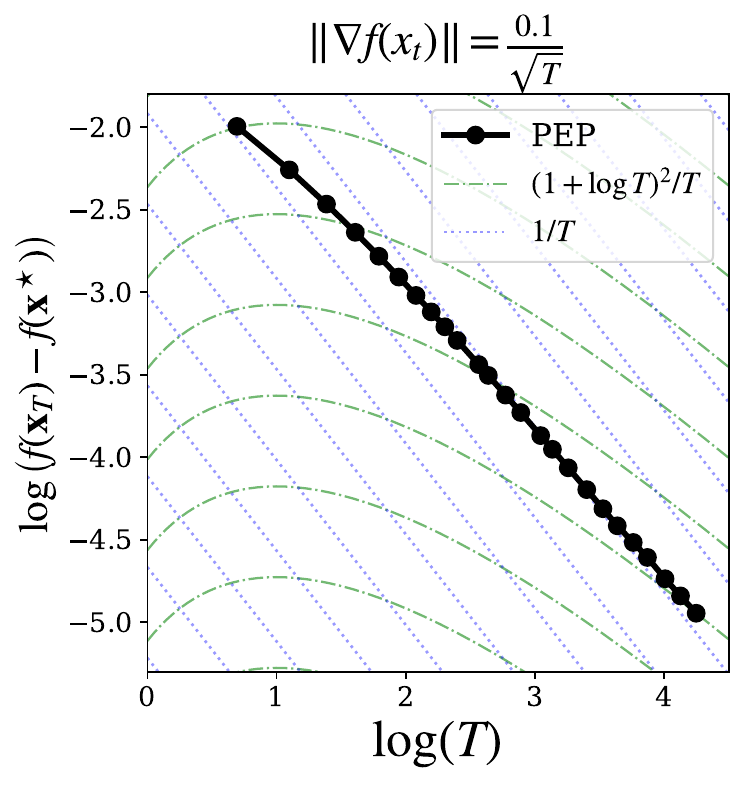}
    \includegraphics[height=4.8cm]{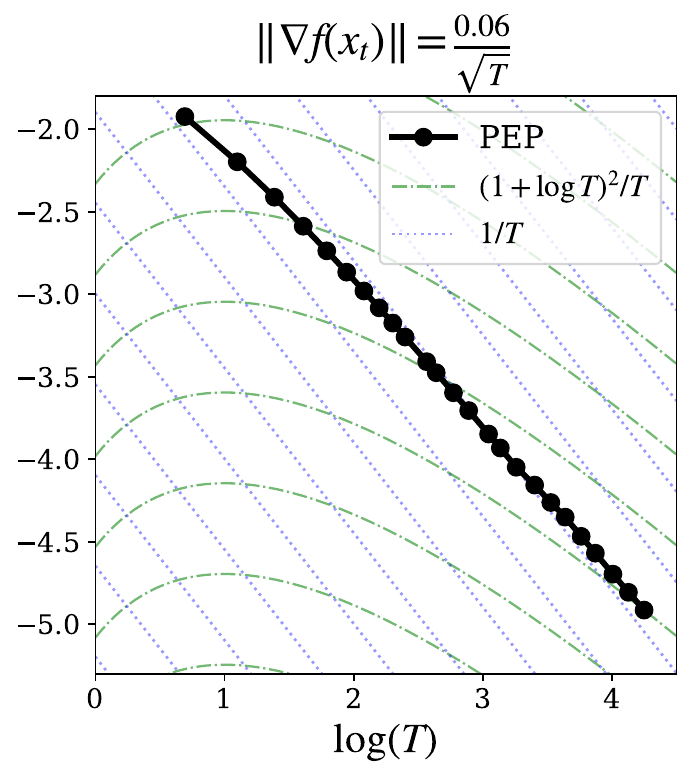}
    \includegraphics[height=4.8cm]{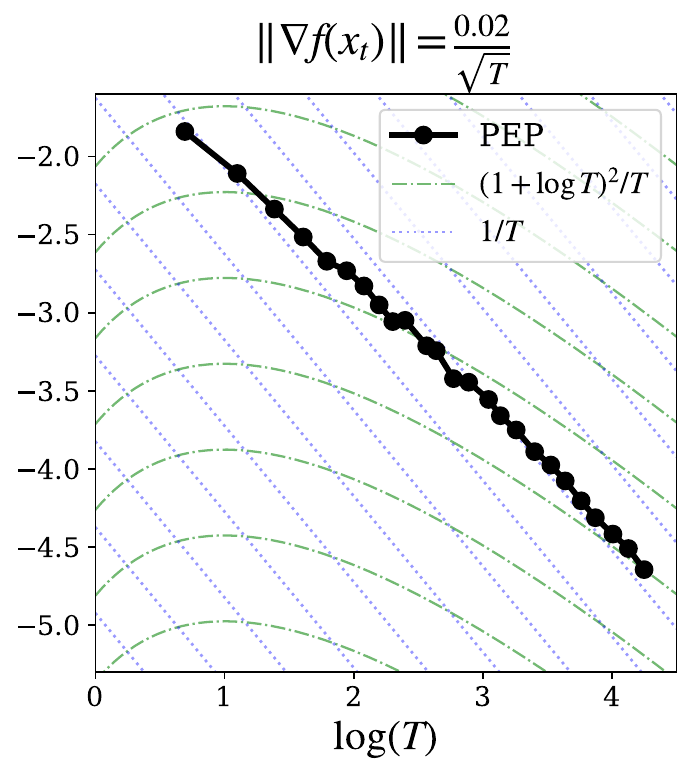}
    \vskip -0.1 in
    \caption{Optimal value of \cref{eq:pep2}, i.e., the worst-case $f(\vx_T) - f(\vx^\star)$ over convex $1$-smooth functions with $\|\vx_1 - \vx^\star\| = 1$ and prescribed gradient norms $g_t = c/\sqrt{T}$, as a function of $T$ (log--log scale) for $c = 0.1$ (left), $c = 0.06$ (middle), and $c = 0.02$ (right). The dotted and dash-dotted lines show the reference decays $1/T$ and $(1 + \log T)^2/T$ up to constant factors. In all three cases, the observed decay is slower than $1/T$.}
    \label{fig:pep}
    \vskip - 0.2 in
\end{figure}

\Cref{fig:pep} shows the results.
For all three values of $c$, the worst-case suboptimality decays slower than $\mathcal{O} ( 1/T )$ over the range of horizons that we could solve, but faster than the $\mathcal{O}(\log^2(T)/T)$ bound of \cref{theorem:convergenec_rate_of_last}.
The latter does not imply that \cref{theorem:convergenec_rate_of_last} is loose, because $g_t \propto 1/\sqrt{T}$ need not be the worst choice of the gradient norms.
We stress that these experiments cover a limited range of $T$ and a hand-picked family $\{ g_t \}_t$: they indicate that the last iterate of normalized gradient descent with a constant stepsize does not attain the $\mathcal{O}(1/T)$ rate of the best iterate, but they neither identify the exact worst-case rate nor prove that the overhead is logarithmic.
Determining the tightness of \cref{theorem:convergenec_rate_of_last} for $\nu > 0$ is left for future work.

\section{Normalized Gradient Descent with a Linearly Decreasing Stepsize}
\label{sec:linearly_decreasing_stepsize}

\Cref{sec:constant_stepsize} shows that, with a constant stepsize, the last iterate of normalized gradient descent suffers a logarithmic overhead which, at least for $\nu = 0$, is unavoidable.
In this section, we show that this overhead disappears if the stepsize linearly decreases.
Specifically, we consider the linearly decreasing stepsize
\begin{align}
\label{eq:ngd_with_adaptive_stepsize}
    \vx_{t+1} = \vx_t - \underbrace{\eta \left( 1 - \frac{t}{T+1} \right)}_{\tilde{\eta}_t} \frac{\nabla f (\vx_t)}{\| \nabla f (\vx_t) \|},
\end{align}
where $\eta > 0$ is the \textit{baseline} stepsize, i.e., the maximum stepsize over the run.
This is the linear-decay schedule widely used for training deep neural networks \citep{bert2019devlin,liu2021swin,defazio2024optimal}; note that it depends only on the iteration counter $t$ and on the horizon $T$, and not on the iterates or gradients.

The following theorem gives the last-iterate rate of normalized gradient descent with the stepsize \cref{eq:ngd_with_adaptive_stepsize} and $\eta = \|\vx_1 - \vx^\star\|/\sqrt{T}$.
The proof is deferred to \cref{sec:proof_of_adaptive_stepsize}.

\begin{theorem}
\label{theorem:convergenec_rate_of_last_with_adaptive_stepsize}
Suppose that Assumptions~\ref{assumption:convex} and \ref{assumption:holder_smoothness} hold.
Then, for any $T \geq 1$, normalized gradient descent with $\tilde{\eta}_t = (1 - \frac{t}{T+1}) \frac{\| \vx_1 - \vx^\star\|}{\sqrt{T}}$ satisfies
\begin{align*}
    f (\vx_T) - f (\vx^\star)
    &\leq \mathcal{O} \left( M_\nu \left( \frac{\| \vx_1 - \vx^\star\|^2}{T} \right)^\frac{1 + \nu}{2} \right).
\end{align*}
\end{theorem}

\paragraph{Discussion.} 
\Cref{theorem:convergenec_rate_of_last_with_adaptive_stepsize} shows that the linearly decreasing stepsize removes the logarithmic overhead of \cref{theorem:convergenec_rate_of_last}: the last iterate attains the same order $\mathcal{O}\bigl(M_\nu (\|\vx_1 - \vx^\star\|^2/T)^{(1+\nu)/2}\bigr)$ as the best-iterate and weighted-average guarantees in \cref{proposition:best_iterate,proposition:weight_averaged_iterate}, and the stepsize requires no knowledge of $\nu$ or $M_\nu$.
For $\nu = 0$, this recovers, up to absolute constants, the result of \citet{zamani2025exact}, who showed that the normalized subgradient method with the linearly decreasing stepsize $\tilde{\eta}_t = \|\vx_1 - \vx^\star\| (1 - \tfrac{t}{T}) / \sqrt{T}$ attains the last-iterate rate $\mathcal{O}(1/\sqrt{T})$ for convex functions with bounded subgradients; their schedule differs from \cref{eq:ngd_with_adaptive_stepsize} only by the shift $T \to T + 1$.
\Cref{theorem:convergenec_rate_of_last_with_adaptive_stepsize} extends this guarantee to the whole H\"older family $\nu \in [0, 1]$.

\subsection{Proof Sketch}

We now explain why a linearly decreasing stepsize removes the logarithmic factor.
As discussed in \cref{sec:proof_sketch_of_constant_stepsize}, this factor comes from the harmonic sum $\sum_{s} \frac{1}{T - s}$ in the bound on $\tilde{\eta}^2 \sum_{t} v_t$ (see \cref{eq:t2}), which is dominated by the iterates $s$ close to $T$.
A stepsize that decreases towards the end of the run reduces the weight of exactly these iterates.

Similar to \cref{sec:proof_sketch_of_constant_stepsize}, we choose $\{ v_t \}_t$ as in \cref{eq:definition_of_v_for_last_iterate}, so that \cref{lemma:zamani} gives
\begin{align*}
    2 \left( f (\vx_T) - f (\vx^\star) \right)
    &\leq v_{0} \| \vx_1 - \vx^\star \|^2 + \sum_{t=1}^T v_t \tilde{\eta}_t^2,
\end{align*}
and it remains to bound $v_0$ and $\sum_{t=1}^T v_t \tilde{\eta}_t^2$ from above.
For the linear schedule, we have the identity $\sum_{s=t+1}^T \tilde{\eta}_s = \tilde{\eta}_t (T - t)/2$, and the Cauchy--Schwarz inequality then yields $v_t \tilde{\eta}_t^2 \leq \frac{4}{(T-t)^2} \sum_{s=t+1}^T \tilde{\eta}_s \| \nabla f(\vx_s) \|$ for every $t < T$; the sums $\sum_{s = t+1}^T \tilde{\eta}_s \| \nabla f (\vx_s) \|$ are in turn controlled by the suboptimality gaps $f(\vx_s) - f(\vx^\star)$ (\cref{lemma:upper_bound_of_e}).

For the constant stepsize, we bounded $f (\vx_t) - f (\vx^\star)$ for every $t$, and the bound necessarily contained a logarithmic factor.
For the linearly decreasing stepsize, we instead show that $f (\vx_t) - f (\vx^\star)$ is bounded by $\mathcal{O} (T^{- \frac{1+\nu}{2}})$ on a sequence of iterates that becomes denser towards the end of the run.
Specifically, we partition the iterations into the blocks
\begin{align*}
    \mathcal{I}_j &\coloneqq \left\{ t \in \{ 1, \dots, T\} \,\,\middle|\,\, 2^{- (j+1)} < 1 - \frac{t}{T+1} \leq 2^{-j} \right\} \quad \text{for all $j \in \{ 0, \dots, J \}$},
\end{align*}
where $J \coloneqq \lceil \log_2 ( \tfrac{T+1}{2 \sqrt{T}} ) \rceil$, so that the stepsize is halved from one block to the next and $|\mathcal{I}_j| \approx (T+1)/2^{j+1}$. This decomposition is similar in spirit to that of \citet{jain2019making}, who halve the stepsize over phases of geometrically decreasing length to remove the logarithmic factor from the last-iterate rate of SGD. In contrast, our stepsize decreases linearly, and the blocks $\mathcal{I}_j$ are used only in the analysis.
We show that each block contains an iterate $s_j \in \mathcal{I}_j$ with
\begin{align*}
    f (\vx_{s_j}) - f (\vx^\star) &\leq 23104 M_\nu \left( \frac{\| \vx_1 - \vx^\star\|^2}{T} \right)^{\frac{1 + \nu}{2}}.
\end{align*}
See \cref{lemma:upper_bound_of_phi} in details.
In words, the best-iterate guarantee of \cref{proposition:best_iterate} ensures that the rate $\mathcal{O}(T^{- \frac{1 + \nu}{2}})$ is attained at least once during the run; with the linearly decreasing stepsize, it is attained once in every block, and the blocks shrink geometrically towards the end of the run.
Since the iterates close to $T$ are the ones that generate the harmonic sum, having such good iterates in every block allows us to bound $v_0$ and $\sum_{t} v_t \tilde{\eta}_t^2$ without a logarithmic factor, which yields \cref{theorem:convergenec_rate_of_last_with_adaptive_stepsize}.
The last at most $\sqrt{T}$ iterations, which are not covered by the blocks, have stepsize $\tilde{\eta}_t$ at most $\|\vx_1 - \vx^\star\|/(T+1)$, and their total contribution is negligible.

\section{Conclusion}
\label{sec:conclusion}
We studied the last-iterate convergence of normalized gradient descent for convex and H\"older-smooth objectives.
With a constant stepsize, the last iterate achieves $\mathcal{O}\bigl((\log^2(T)/T)^{(1+\nu)/2}\bigr)$, which is slower than the known best-iterate and weighted-average rates by a logarithmic factor; this factor is unavoidable for $\nu = 0$, and our PEP-based numerical results indicate that the last iterate is also slower than $\mathcal{O}(1/T)$ in the smooth case.
With a linearly decreasing stepsize, the logarithmic factor disappears and the last iterate attains the rate $\mathcal{O}\bigl(T^{-(1+\nu)/2}\bigr)$ without knowledge of the H\"older parameters.

\section*{Acknowledgement}
Part of this work was conducted while YT was visiting MBZUAI.

\section*{AI use statement}

In this work, we used generative AI tools for language editing and proofreading of the manuscript, including improving grammar, clarity, and readability.
We have not used generative AI tools to develop the main research ideas,
derive theoretical results, or write the code in our experiments.
We have reviewed all AI-assisted work. We take responsibility for the final content of this work, including text, claims or artifacts produced with the aid of generative AI.

\bibliography{ref}
\bibliographystyle{iclr2027_conference}

\newpage
\appendix
\crefalias{section}{appendix}

\section{Importance of Convergence Analysis in the Convex Setting}
\label{sec:importance_of_convexity}
In this paper, we assume that the objective is convex.
In practical applications, most loss functions are non-convex; nevertheless, analysis in the convex setting remains highly important.
One reason is that convergence analyses under non-convex assumptions are often overly pessimistic. Since optimization theory typically studies worst-case functions, if one allows arbitrary non-convex functions, the resulting guarantees can become too conservative to accurately reflect the behavior observed in practical problems.
For example, under the assumptions of non-convexity and smoothness, gradient descent is known to be an optimal first-order optimization method for finding approximate stationary points, and its convergence rate cannot be further improved in general \citep{carmon2021lower}.
Furthermore, several recent works report that predictions of convex optimization theory match the behavior of neural network training surprisingly well, e.g., regarding learning-rate schedules \citep{defazio2024the,schaipp2025the}.
We believe that our analysis of normalized gradient descent in the convex setting provides useful insight beyond the convex regime and may serve as a basis for understanding its empirical effectiveness in broader non-convex problems.

\section{Useful Lemmas}

\begin{lemma}[\citet{nesterov2016universal}]
\label{lemma:holder_smooth}
Suppose that \cref{assumption:holder_smoothness} holds. Then, for any $\vx, \vy \in \mathbb{R}^d$, it holds that
\begin{align*}
    f(\vy) \leq f (\vx) + \langle \nabla f (\vx), \vy - \vx \rangle + \frac{M_\nu}{1 + \nu} \left\| \vx - \vy \right\|^{1+\nu}.
\end{align*}
\end{lemma}

\begin{lemma}[Lemma A.4 in \citet{gorbunov2024high}]
\label{lemma:bound_of_grad_norm}
Suppose that \cref{assumption:holder_smoothness} holds. Then, for any $\vx \in \mathbb{R}^d$, it holds that
\begin{align*}
    \| \nabla f (\vx) \| \leq \left( \frac{1 + \nu}{\nu} \right)^\frac{\nu}{1 + \nu} M_\nu^\frac{1}{1 + \nu} \left( f (\vx) - f (\vx^\star) \right)^\frac{\nu}{1 + \nu},
\end{align*}
where for $\nu=0$, we use $\lim_{\nu \rightarrow 0} ( \tfrac{1 + \nu}{\nu} )^{\frac{\nu}{1 + \nu}} = 1$.
\end{lemma}

\begin{lemma}
\label{lemma:self_improving_bound}
For any $x, a, b \geq 0$ and $\nu \in [0,1]$, if
\begin{align*}
    x \leq a x^\frac{\nu}{1 + \nu} + b,
\end{align*}
then it holds that
\begin{align*}
    x \leq a^{1 + \nu} + 2 b.
\end{align*}
\end{lemma}
\begin{proof}
From Young's inequality, we have
\begin{align*}
    a x^\frac{\nu}{1 + \nu} \leq \frac{\nu}{1 + \nu} x + \frac{a^{1 + \nu}}{1 + \nu}.
\end{align*}
Substituting this bound into the assumption gives $\frac{x}{1 + \nu} \leq \frac{a^{1+\nu}}{1+\nu} + b$, i.e., $x \leq a^{1+\nu} + (1+\nu) b \leq a^{1 + \nu} + 2 b$.
\end{proof}

\zamaniLemma*
\begin{proof}
Let $\vz_0 = \vx^\star$ and $\vz_t$ be defined recursively as follows:
\begin{align*}
    \vz_t = \left( 1 - \frac{v_{t-1}}{v_t} \right) \vx_t + \frac{v_{t-1}}{v_t} \vz_{t-1}.
\end{align*}

We have
\begin{align*}
    \| \vx_{t+1} - \vz_t \|^2
    &= \| \vx_t - \vz_t \|^2 - 2 \eta_t \langle \nabla f (\vx_t), \vx_t - \vz_t \rangle + \tilde{\eta}_t^2 \\
    &\leq \| \vx_t - \vz_t \|^2 - 2 \eta_t \left( f (\vx_t) - f (\vz_t) \right) + \tilde{\eta}_t^2 \\
    &\leq \left( \frac{v_{t-1}}{v_t} \right) \| \vx_t - \vz_{t-1} \|^2 - 2 \eta_t \left( f (\vx_t) - f (\vz_t) \right) + \tilde{\eta}_t^2,
\end{align*}
where we use \cref{assumption:convex} for the first inequality and $v_{t-1} \leq v_t$ for the last inequality.
Thus, we get
\begin{align*}
    2 \eta_t v_t \left( f (\vx_t) - f (\vz_t) \right)
    &\leq v_{t-1} \| \vx_t - \vz_{t-1} \|^2 - v_t \| \vx_{t+1} - \vz_t \|^2 + v_t \tilde{\eta}_t^2.
\end{align*}
Unrolling the definition of $\vz_t$ gives $\vz_t = \sum_{s=1}^t \frac{v_s - v_{s-1}}{v_t} \vx_s + \frac{v_0}{v_t} \vx^\star$, a convex combination since $\{v_t\}_t$ is nondecreasing and positive; hence, by the convexity of $f$, we have
\begin{align*}
    f (\vz_t) \leq \sum_{s=1}^t \frac{v_s - v_{s-1}}{v_t} f (\vx_s) + \frac{v_0}{v_t} f (\vx^\star).
\end{align*}
Multiplying this inequality by $2 \eta_t v_t$, summing the resulting bounds and the previous display over $t = 1, \dots, T$, exchanging the order of summation, and using $\vz_0 = \vx^\star$ and $\| \vx_{T+1} - \vz_T \|^2 \geq 0$, we obtain
\begin{align*}
    2 \sum_{t=1}^T \left[ \eta_t v_t - \left( \sum_{s=t}^T \eta_s \right) (v_t - v_{t-1}) \right] \left( f (\vx_t) - f (\vx^\star) \right)
    &\leq 2 \sum_{t=1}^T \eta_t v_t \left( f (\vx_t) - f (\vz_t) \right) \\
    &\leq v_{0} \| \vx_1 - \vx^\star \|^2 + \sum_{t=1}^T v_t \tilde{\eta}_t^2.
\end{align*}
\end{proof}

\section{Proof of \Cref{theorem:convergenec_rate_of_last}}
\label{sec:proof_of_last}

\LemmaaveragedRateWithConstant*
\begin{proof}
We prove the statement by induction on $t$.
For $t = 1$, \cref{lemma:holder_smooth} with $\vx = \vx^\star$ and $\nabla f(\vx^\star) = 0$ gives
\begin{align*}
    f (\vx_1) - f (\vx^\star) \leq \frac{M_\nu}{1 + \nu} \| \vx_1 - \vx^\star\|^{1 + \nu},
\end{align*}
which is at most the right-hand side of \cref{eq:last_iterate_for_every_t} for $t = 1$ since $200 \geq \frac{1}{1+\nu}$.
Thus, the statement holds for $t=1$.

Now let $T \geq 2$, assume that $\tilde{\eta} \leq \frac{\| \vx_1 - \vx^\star \|}{\sqrt{T}}$, and assume that \cref{eq:last_iterate_for_every_t} holds for every $t \leq T-1$; note that the induction hypothesis is applicable to every $t \le T - 1$, because $\tilde{\eta} \leq \frac{\| \vx_1 - \vx^\star \|}{\sqrt{T}} \leq \frac{\| \vx_1 - \vx^\star \|}{\sqrt{t}}$.
We show that \cref{eq:last_iterate_for_every_t} holds for $t = T$. 
Let $\{ v_t \}_{t=0}^T$ be defined as follows:
\begin{align*}
    v_T &\coloneqq \frac{1}{\eta_T} \\
    v_t &\coloneqq \frac{1}{\sum_{s=t+1}^T \eta_s} \quad \text{(for all $t = 0, 1, \dots, T-1$)}.
\end{align*}
Note that $\{ v_t \}_{t=0}^T$ is positive and nondecreasing (in fact, $v_{T-1} = v_T$).
In this case, we have
\begin{align*}
    \eta_T v_T &= 1,
\end{align*}
and for any $t \in \{ 1, 2, \dots, T-1\}$, we have
\begin{align*}
        \eta_t v_t - \left( \sum_{s=t}^T \eta_s \right) (v_t - v_{t-1})
    &=  - \left( \sum_{s=t+1}^T \eta_s \right) v_t + \left( \sum_{s=t}^T \eta_s \right) v_{t-1}
    = 0.
\end{align*}
Thus, from Lemma~\ref{lemma:zamani}, we have
\begin{align*}
    2 \left( f (\vx_T) - f (\vx^\star) \right)
    &\leq v_{0} \| \vx_1 - \vx^\star \|^2 + \tilde{\eta}^2 \sum_{t=1}^T v_t.
\end{align*}

We first bound $v_0$. Dropping the term $s = T$, we have
\begin{align*}
    v_0 &= \frac{1}{\tilde{\eta} \sum_{s=1}^T \frac{1}{\| \nabla f (\vx_s) \|}}
    \leq \frac{1}{\tilde{\eta} \sum_{s=1}^{T-1} \frac{1}{\| \nabla f (\vx_s) \|}}.
\end{align*}
For any $s \leq T-1$, the induction hypothesis \cref{eq:last_iterate_for_every_t} and the inequality $\bigl( \frac{\| \vx_1 - \vx^\star \|^2}{s} \bigr)^{\frac{1+\nu}{2}} \leq \bigl( \frac{\| \vx_1 - \vx^\star \|^2}{\tilde{\eta} s} \bigr)^{1 + \nu}$, which follows from $\tilde{\eta} \leq \frac{\| \vx_1 - \vx^\star\|}{\sqrt{s}}$, give $f(\vx_s) - f(\vx^\star) \leq 1400 M_\nu \bigl( \frac{\| \vx_1 - \vx^\star \|^2}{\tilde{\eta} s} \bigr)^{1 + \nu} (1 + \log(s))^{1+\nu}$.
Combining this with \cref{lemma:bound_of_grad_norm}, $\bigl( \frac{1+\nu}{\nu} \bigr)^{\frac{\nu}{1+\nu}} \leq 2$, and $1400^{\frac{\nu}{1+\nu}} \leq \sqrt{1400} < 40$, it holds for any $s \leq T-1$ that
\begin{align*}
    \left\| \nabla f (\vx_s) \right\| \leq 80 M_\nu \left( \frac{\| \vx_1 - \vx^\star\|^2}{\tilde{\eta} s} \right)^\nu (1 + \log (s))^\nu.
\end{align*}
Applying the above inequality  and $\sum_{s=1}^{T-1} s^\nu \geq \frac{T^{1+\nu}}{4}$ for $T \geq 2$, we obtain
\begin{align*}
    v_0 &\leq \frac{80 M_\nu \| \vx_1 - \vx^\star\|^{2 \nu}}{\tilde{\eta}^{1 + \nu}} (1 + \log (T))^\nu \frac{1}{\sum_{s=1}^{T-1} s^\nu} \\
    &\leq 320 (1 + \log (T))^\nu M_\nu \left( \frac{\| \vx_1 - \vx^\star\|^{2\nu}}{\tilde{\eta}^{1 + \nu} T^{1 + \nu}} \right).
\end{align*}

Next, we bound $\sum_{t=1}^T v_t$:
\begin{align*}
    \sum_{t=1}^T v_t
    &= v_T + \frac{1}{\tilde{\eta}} \sum_{t=1}^{T-1} \frac{1}{\sum_{s=t+1}^T \frac{1}{\| \nabla f (\vx_s) \|}} \\
    &\leq v_T + \frac{1}{\tilde{\eta}} \sum_{t=1}^{T-1} \frac{1}{(T - t)^2} \sum_{s=t+1}^T \| \nabla f (\vx_s) \| \\
    &\leq v_T + \frac{1}{\tilde{\eta}} \left( \frac{1 + \nu}{\nu} \right)^\frac{\nu}{1 + \nu} M_\nu^\frac{1}{1 + \nu} \sum_{t=1}^{T-1} \frac{1}{(T - t)^2} \sum_{s=t+1}^T  \left( f (\vx_s) - f (\vx^\star) \right)^\frac{\nu}{1 + \nu} \\
    &\leq \frac{M_\nu^\frac{1}{1 + \nu}}{\tilde{\eta}} \left( \frac{1 + \nu}{\nu} \right)^\frac{\nu}{1 + \nu} \left( f (\vx_T) - f (\vx^\star) \right)^\frac{\nu}{1 + \nu}\\
    &\quad + \frac{M_\nu^\frac{1}{1 + \nu}}{\tilde{\eta}} \left( \frac{1 + \nu}{\nu} \right)^\frac{\nu}{1 + \nu} \sum_{t=1}^{T-1} \frac{1}{(T - t)^2} \sum_{s=t+1}^T  \left( f (\vx_s) - f (\vx^\star) \right)^\frac{\nu}{1 + \nu} \\
    &\leq \frac{M_\nu^\frac{1}{1 + \nu}}{\tilde{\eta}} \left( \frac{1 + \nu}{\nu} \right)^\frac{\nu}{1 + \nu} \left( f (\vx_T) - f (\vx^\star) \right)^\frac{\nu}{1 + \nu} \left( 1 + \sum_{t=1}^{T-1} \frac{1}{(T - t)^2} \right) \\
    &\quad + \frac{M_\nu^\frac{1}{1 + \nu}}{\tilde{\eta}} \left( \frac{1 + \nu}{\nu} \right)^\frac{\nu}{1 + \nu} \sum_{t=1}^{T-1} \frac{1}{(T - t)^2} \sum_{s=t+1}^{T-1}  \left( f (\vx_s) - f (\vx^\star) \right)^\frac{\nu}{1 + \nu} \\
    &\leq \frac{6 M_\nu^\frac{1}{1 + \nu}}{\tilde{\eta}}  \left( f (\vx_T) - f (\vx^\star) \right)^\frac{\nu}{1 + \nu}
    + \frac{2 M_\nu^\frac{1}{1 + \nu}}{\tilde{\eta}}  \sum_{t=1}^{T-1} \frac{1}{(T - t)^2} \sum_{s=t+1}^{T-1}  \left( f (\vx_s) - f (\vx^\star) \right)^\frac{\nu}{1 + \nu} \\
    &= \frac{6 M_\nu^\frac{1}{1 + \nu}}{\tilde{\eta}}  \left( f (\vx_T) - f (\vx^\star) \right)^\frac{\nu}{1 + \nu}
    + \frac{2 M_\nu^\frac{1}{1 + \nu}}{\tilde{\eta}}  \sum_{s=2}^{T-1}  \left( f (\vx_s) - f (\vx^\star) \right)^\frac{\nu}{1 + \nu} \sum_{t=1}^{s-1} \frac{1}{(T - t)^2},
\end{align*}
where we use the fact that the harmonic mean is smaller than or equal to the arithmetic mean for the first inequality, \cref{lemma:bound_of_grad_norm} (a consequence of \cref{assumption:holder_smoothness}) for the second and third inequalities, and $\bigl( \frac{1+\nu}{\nu} \bigr)^{\frac{\nu}{1+\nu}} \leq 2$ together with $1 + \sum_{t=1}^{T-1} \frac{1}{(T-t)^2} \leq 1 + \frac{\pi^2}{6} \leq 3$ for the fifth inequality; the last step exchanges the order of summation.
Then, we have

\begin{align*}
    &\frac{2 M_\nu^\frac{1}{1 + \nu}}{\tilde{\eta}}  \sum_{s=2}^{T-1}  \left( f (\vx_s) - f (\vx^\star) \right)^\frac{\nu}{1 + \nu} \sum_{t=1}^{s-1} \frac{1}{(T - t)^2} \\
    &= \frac{2 M_\nu^\frac{1}{1 + \nu}}{\tilde{\eta}} \underbrace{\sum_{s=2}^{\floor{\frac{T-1}{2}}}  \left( f (\vx_s) - f (\vx^\star) \right)^\frac{\nu}{1 + \nu} \sum_{t=1}^{s-1} \frac{1}{(T - t)^2}}_{\mathcal{T}_1}
    \\
    &\quad + \frac{2 M_\nu^\frac{1}{1 + \nu}}{\tilde{\eta}}  \underbrace{\sum_{s=\floor{\frac{T-1}{2}} + 1}^{T-1}  \left( f (\vx_s) - f (\vx^\star) \right)^\frac{\nu}{1 + \nu} \sum_{t=1}^{s-1} \frac{1}{(T - t)^2}}_{\mathcal{T}_2}.
\end{align*}

Using $s \leq \frac{T-1}{2}$, so that $T - t \geq \frac{T}{2}$ for all $t \leq s-1$, we have
\begin{align*}
    \mathcal{T}_1
    &\leq \sum_{s=2}^{\floor{\frac{T-1}{2}}}  \left( f (\vx_s) - f (\vx^\star) \right)^\frac{\nu}{1 + \nu} \frac{4 s}{T^2} \\
    &\leq 4 M_\nu^\frac{\nu}{1 + \nu} \frac{(1 + \log (T))^\nu}{T^2} \sum_{s=2}^{\floor{\frac{T-1}{2}}} \left( 35 \left( \frac{\| \vx_1 - \vx^\star \|^{2}}{\tilde{\eta}} \right)^{\nu} s^{1 - \nu}
    + 15 \| \vx_1 - \vx^\star\|^\nu s^{1 - \frac{\nu}{2}} \right) \\
    &\leq 4 M_\nu^\frac{\nu}{1 + \nu} \frac{(1 + \log (T))^\nu}{T^2} \left( 35 \left( \frac{\| \vx_1 - \vx^\star \|^{2 \nu}}{\tilde{\eta}^\nu } \right) T^{2 - \nu}
    + 15 \| \vx_1 - \vx^\star\|^\nu T^{2 - \frac{\nu}{2}} \right) \\
    &\leq 200 M_\nu^\frac{\nu}{1 + \nu} (1 + \log (T))^\nu \left( \frac{\| \vx_1 - \vx^\star \|^{2 \nu}}{\tilde{\eta}^{\nu} T^\nu } \right),
\end{align*}
where we use the induction hypothesis together with $(a + b)^{p} \leq a^{p} + b^{p}$ for $p = \frac{\nu}{1+\nu} \in [0, \frac12]$ in the second inequality, $\sum_{s=2}^{\lfloor (T-1)/2 \rfloor} s^{\alpha} \leq T^{1 + \alpha}$ for $\alpha \in [0, 1]$ in the third inequality, and $\tilde{\eta} \leq \tfrac{\| \vx_1 - \vx^\star\|}{\sqrt{T}} $, i.e., $\| \vx_1 - \vx^\star \|^{\nu} T^{-\nu/2} \leq \bigl( \frac{\| \vx_1 - \vx^\star \|^2}{\tilde{\eta} T} \bigr)^{\nu}$, in the last inequality.
Using $s \geq \floor{\frac{T-1}{2}} + 1$, i.e., $s \geq \frac{T}{2}$, and $T \geq 2$, we have
\begin{align*}
    \mathcal{T}_2
    &\leq  (1 + \log (T))^{\nu}  M_\nu^\frac{\nu}{1 + \nu} \sum_{s=\floor{\frac{T-1}{2}}}^{T-1}  \left( 70 \left( \frac{\| \vx_1 - \vx^\star \|^{2}}{\tilde{\eta} T} \right)^{\nu}
    + 30 \left( \frac{\| \vx_1 - \vx^\star\|^2}{T} \right)^\frac{\nu}{2} \right) \sum_{t=1}^{s-1} \frac{1}{(T - t)^2} \\
    &\leq  (1 + \log (T))^{\nu}  M_\nu^\frac{\nu}{1 + \nu} \sum_{s=\floor{\frac{T-1}{2}}}^{T-1}  \left( 70 \left( \frac{\| \vx_1 - \vx^\star \|^{2}}{\tilde{\eta} T} \right)^{\nu}
    + 30 \left( \frac{\| \vx_1 - \vx^\star\|^2}{T} \right)^\frac{\nu}{2} \right) \frac{1}{T - s} \\
    &\leq  (1 + \log (T))^{1 + \nu}  M_\nu^\frac{\nu}{1 + \nu}  \left( 70 \left( \frac{\| \vx_1 - \vx^\star \|^{2}}{\tilde{\eta} T} \right)^{\nu}
    + 30 \left( \frac{\| \vx_1 - \vx^\star\|^2}{T} \right)^\frac{\nu}{2} \right),
\end{align*}
where the first inequality uses the induction hypothesis as for $\mathcal{T}_1$ together with $s \geq \frac{T}{2}$ and extends the range of summation to $s \geq \lfloor \frac{T-1}{2} \rfloor$, the second inequality uses $\sum_{t=1}^{s-1} \frac{1}{(T-t)^2} \leq \sum_{u = T-s+1}^{\infty} \frac{1}{u^2} \leq \frac{1}{T - s}$, and the last inequality uses $\sum_{s = \lfloor (T-1)/2 \rfloor}^{T-1} \frac{1}{T-s} \leq 1 + \log(T)$.
Combining the above inequalities, we get
\begin{align*}
    2 \left( f (\vx_T) - f (\vx^\star) \right)
    &\leq 6 \tilde{\eta} M_\nu^\frac{1}{1 + \nu} \left( f (\vx_T) - f (\vx^\star) \right)^\frac{\nu}{1 + \nu} \\
    &+ (1 + \log (T))^{1 + \nu}  M_\nu  \left( 860 \left( \frac{\| \vx_1 - \vx^\star \|^{2}}{\tilde{\eta} T} \right)^{1 + \nu}
    + 60 \left( \frac{\| \vx_1 - \vx^\star\|^2}{T} \right)^\frac{1 + \nu}{2} \right).
\end{align*}
Then, using $\tilde{\eta} \leq \tfrac{\| \vx_1 - \vx^\star\|}{\sqrt{T}}$ and $T \geq 1$, we obtain
\begin{align*}
    f (\vx_T) - f (\vx^\star)
    &\leq \frac{4 \| \vx_1 - \vx^\star\|^2}{T \tilde{\eta}} M_\nu^\frac{1}{1 + \nu}  \left( f (\vx_T) - f (\vx^\star) \right)^\frac{\nu}{1 + \nu} \\
    &+ (1 + \log (T))^{1 + \nu}  M_\nu  \left( 430 \left( \frac{\| \vx_1 - \vx^\star \|^{2}}{\tilde{\eta} T} \right)^{1 + \nu}
    + 30 \left( \frac{\| \vx_1 - \vx^\star\|^2}{T} \right)^\frac{1 + \nu}{2} \right).
\end{align*}
From \cref{lemma:self_improving_bound}, we obtain
\begin{align*}
    f (\vx_T) - f (\vx^\star)
    &\leq (1 + \log (T))^{1 + \nu}  M_\nu  \left( 876 \left( \frac{\| \vx_1 - \vx^\star \|^{2}}{\tilde{\eta} T} \right)^{1 + \nu}
    + 60 \left( \frac{\| \vx_1 - \vx^\star\|^2}{T} \right)^\frac{1 + \nu}{2} \right).
\end{align*}
This implies \cref{eq:last_iterate_for_every_t} for $t = T$, which completes the induction.
Finally, substituting $\tilde{\eta} = \frac{\| \vx_1 - \vx^\star \|}{\sqrt{T}}$ and $t = T$ into \cref{eq:last_iterate_for_every_t} proves \cref{theorem:convergenec_rate_of_last} with the absolute constant $1400$; we have not attempted to optimize this constant.
\end{proof}

\section{Proof of \Cref{theorem:convergenec_rate_of_last_with_adaptive_stepsize}}
\label{sec:proof_of_adaptive_stepsize}
\paragraph{Notation.}
Throughout this section, $\tilde{\eta}_t = ( 1 - \tfrac{t}{T+1} ) \frac{\| \vx_1 - \vx^\star\|}{\sqrt{T}}$ for $t \geq 1$, and we also set $\tilde{\eta}_0 \coloneqq \frac{\| \vx_1 - \vx^\star \|}{\sqrt{T}}$.
We use the following notation:
\begin{align*}
    f^\star &\coloneqq f(\vx^\star), \\
    \Phi(\vx) &\coloneqq (f(\vx) - f^\star)^{\frac{1}{1+\nu}}, \\
    J &\coloneqq \left\lceil \log_2 \left( \frac{T+1}{2 \sqrt{T}} \right) \right\rceil, \\
    \mathcal{I}_j &\coloneqq \left\{ t \in \{ 1, \dots, T\} \mid 2^{- (j+1)} < 1 - \frac{t}{T+1} \leq 2^{-j} \right\} \quad \text{for all $j \in \{ 0, \dots, J \}$}.
\end{align*}
By the definition of $J$, we have $\frac{\sqrt{T}}{T+1} < 2^{-J} \leq \frac{2 \sqrt{T}}{T+1}$.
The blocks $\mathcal{I}_0, \dots, \mathcal{I}_J$ are consecutive intervals of integers: $\mathcal{I}_0 = \{1, \dots, \lfloor T/2 \rfloor\}$, every element of $\mathcal{I}_{j-1}$ is smaller than every element of $\mathcal{I}_j$, and, for $T \geq 16$, $\max \mathcal{I}_J \leq T - 2$ and $|\mathcal{I}_J| \geq 1$, because $(T+1) 2^{-(J+1)} > \frac{\sqrt{T}}{2} \geq 2$.
The iterations $t > \max \mathcal{I}_J$ that are not covered by the blocks satisfy $T + 1 - t \leq (T+1) 2^{-(J+1)} \leq \sqrt{T}$ and $\tilde{\eta}_t \leq \frac{\| \vx_1 - \vx^\star \|}{T + 1}$.

\begin{lemma}
\label{lemma:difference_of_function_value}
Suppose that \cref{assumption:holder_smoothness,assumption:convex} hold.
If $f(\vx) \geq f (\vy)$ and $\nabla f(\vx) \neq 0$, it holds that
\begin{align*}
    \Phi (\vx) - \Phi (\vy)
    \leq 2 M_\nu^\frac{1}{1 + \nu} \left\langle \frac{\nabla f (\vx)}{\| \nabla f (\vx)\|}, \vx - \vy \right\rangle,
\end{align*}
where $\Phi (\vx) \coloneqq \left( f (\vx) - f^\star \right)^\frac{1}{1 + \nu}$.
In particular, when $\vy = \vx^\star$ we have
\begin{align*}
    \Phi (\vx)
    \leq 2 M_\nu^\frac{1}{1 + \nu} \left\langle \frac{\nabla f (\vx)}{\| \nabla f (\vx)\|}, \vx - \vx^\star \right\rangle.
\end{align*}
\end{lemma}
\begin{proof}
Using $f(\vx) \geq f (\vy)$, we have
\begin{align*}
    \left( f (\vx) - f^\star \right)^\frac{1}{1 + \nu} -     \left( f (\vy) - f^\star \right)^\frac{1}{1 + \nu}
    &\leq \left( f (\vx) - f^\star \right)^\frac{- \nu}{1 + \nu} \left( f (\vx)  - f (\vy) \right) \\
    &\leq \left( f (\vx) - f^\star \right)^\frac{- \nu}{1 + \nu} \left\langle \nabla f (\vx), \vx - \vy \right\rangle \\
    &\leq 2 M_\nu^\frac{1}{1 + \nu} \left\langle \frac{\nabla f (\vx)}{\| \nabla f (\vx) \|}, \vx - \vy \right\rangle,
\end{align*}
where the first inequality uses $a^{p} - b^{p} \leq a^{p-1} (a - b)$ for $a \geq b \geq 0$ and $p = \frac{1}{1+\nu} \in [\frac12, 1]$, the second inequality uses \cref{assumption:convex}, and the last inequality uses $\langle \nabla f(\vx), \vx - \vy \rangle \geq f(\vx) - f(\vy) \geq 0$ and \cref{lemma:bound_of_grad_norm} in the form $(f(\vx) - f^\star)^{-\frac{\nu}{1+\nu}} \| \nabla f(\vx) \| \leq 2 M_\nu^{\frac{1}{1+\nu}}$.
\end{proof}

\begin{lemma}
\label{lemma:difference_of_function_values}
Suppose that \cref{assumption:holder_smoothness,assumption:convex} hold.
Let $1 \leq s \leq p \leq q \leq T$.
There exists $r \in \{ p, \dots, q \}$ such that
\begin{align*}
    \Phi (\vx_r) \leq \Phi (\vx_s) + M_\nu^\frac{1}{1 + \nu} \frac{\sum_{t=s}^q \tilde{\eta}_t^2}{\sum_{t=p}^q \tilde{\eta}_t},
\end{align*}
where $\Phi (\vx) \coloneqq \left( f (\vx) - f^\star \right)^\frac{1}{1 + \nu}$.
\end{lemma}
\begin{proof}
Choose $s^\star \in \{ s, \dots, p \}$ minimizing $f (\vx_t)$ on this interval.

If some $r \in \{ p, \dots, q \}$ satisfies $f (\vx_r) < f (\vx_{s^\star})$, the statement holds, since then $\Phi(\vx_r) < \Phi(\vx_{s^\star}) \leq \Phi(\vx_s)$.
Otherwise, $f (\vx_t) \geq f (\vx_{s^\star})$ for all $t \in \{ p, \dots, q\}$.
For any $t \in \{ p, \dots, q\}$, \cref{lemma:difference_of_function_value} implies that
\begin{align*}
    \Phi (\vx_t) - \Phi (\vx_{s^\star})
    \leq 2 M_\nu^\frac{1}{1 + \nu} \left\langle \frac{\nabla f (\vx_t)}{\| \nabla f (\vx_t)\|}, \vx_t - \vx_{s^\star} \right\rangle.
\end{align*}
The update rule of normalized gradient descent gives
\begin{align*}
    \| \vx_{t+1} - \vx_{s^\star} \|^2
    = \| \vx_t - \vx_{s^\star} \|^2
    - 2 \tilde{\eta}_t \left\langle \frac{\nabla f (\vx_t)}{\| \nabla f (\vx_t) \|}, \vx_t - \vx_{s^\star} \right\rangle
    + \tilde{\eta}_t^2.
\end{align*}
Summing this identity from $t = s^\star$ to $q$ and dropping the nonnegative term $\| \vx_{q+1} - \vx_{s^\star} \|^2$ gives $2 \sum_{t=s^\star}^{q} \tilde{\eta}_t \langle \frac{\nabla f (\vx_t)}{\| \nabla f (\vx_t) \|}, \vx_t - \vx_{s^\star} \rangle \leq \sum_{t=s^\star}^{q} \tilde{\eta}_t^2$.
Moreover, for $t \in \{ s^\star, \dots, p - 1 \}$ we have $f(\vx_t) \geq f(\vx_{s^\star})$ by the choice of $s^\star$, so \cref{assumption:convex} gives $\langle \nabla f(\vx_t), \vx_t - \vx_{s^\star} \rangle \geq f(\vx_t) - f(\vx_{s^\star}) \geq 0$; hence, the terms with $t < p$ in the sum on the left-hand side are nonnegative and can be dropped.
Combining these facts with the previous inequality yields
\begin{align*}
    \sum_{t=p}^q \tilde{\eta}_t \left( \Phi (\vx_t) - \Phi (\vx_{s^\star}) \right)
    &\leq 2 M_\nu^\frac{1}{1 + \nu} \sum_{t=s^\star}^q \tilde{\eta}_t \left\langle \frac{\nabla f (\vx_t)}{\| \nabla f (\vx_t)\|}, \vx_t - \vx_{s^\star} \right\rangle \\
    &\leq M_\nu^\frac{1}{1 + \nu} \sum_{t=s^\star}^q \tilde{\eta}_t^2 \\
    &\leq M_\nu^\frac{1}{1 + \nu} \sum_{t=s}^q \tilde{\eta}_t^2,
\end{align*}
where we use $s \leq s^\star \leq p$.
Choosing $r \in \{ p, \dots, q\}$ minimizing $\Phi (\vx_t)$ and using $\bigl( \Phi(\vx_r) - \Phi(\vx_{s^\star}) \bigr) \sum_{t=p}^{q} \tilde{\eta}_t \leq \sum_{t=p}^q \tilde{\eta}_t ( \Phi (\vx_t) - \Phi (\vx_{s^\star}) )$ yields
\begin{align*}
    \Phi (\vx_r) - \Phi (\vx_{s^\star})
    &\leq M_\nu^\frac{1}{1 + \nu} \frac{\sum_{t=s}^q \tilde{\eta}_t^2}{\sum_{t=p}^q \tilde{\eta}_t}.
\end{align*}
Since $\Phi(\vx_{s^\star}) \leq \Phi(\vx_s)$, this concludes the proof.
\end{proof}

\begin{lemma}
\label{lemma:upper_bound_of_phi}
Suppose that \cref{assumption:holder_smoothness,assumption:convex} hold, $T \geq 16$, and $\tilde{\eta}_t = ( 1 - \tfrac{t}{T+1} ) \frac{\| \vx_1 - \vx^\star\|}{\sqrt{T}}$. For any $j \in \{ 0, \dots, J \}$, there exists $s_j \in \mathcal{I}_j$ such that
\begin{align*}
    f (\vx_{s_j}) - f^\star &\leq 23104 M_\nu \left( \frac{\| \vx_1 - \vx^\star\|^2}{T} \right)^{\frac{1 + \nu}{2}}.
\end{align*}
\end{lemma}
\begin{proof}
For $t \in \mathcal{I}_j$, we have, by the definition of $\mathcal{I}_j$
\begin{align}
    \label{eq:regime_of_sum_of_stepsize}
    \frac{1}{2^{j+1}} \frac{\| \vx_1 - \vx^\star\|}{\sqrt{T}} &< \tilde{\eta}_t \leq \frac{1}{2^{j}} \frac{\| \vx_1 - \vx^\star\|}{\sqrt{T}}, \\
    \label{eq:regime_of_sum_of_squared_stepsize}
    \frac{T + 1}{2^{j+1}} -1 &\leq | \mathcal{I}_j | \leq \frac{T + 1}{2^{j+1}} + 1.
\end{align}

From \cref{lemma:difference_of_function_value}, we obtain
\begin{align*}
    \sum_{t \in \mathcal{I}_0} \tilde{\eta}_t \Phi (\vx_t)
    &\leq 2 M_\nu^\frac{1}{1 + \nu} \sum_{t \in \mathcal{I}_0}  \tilde{\eta}_t \left\langle \frac{\nabla f (\vx_t)}{\| \nabla f (\vx_t)\|}, \vx_t - \vx^\star \right\rangle \\
    &\leq 2 M_\nu^\frac{1}{1 + \nu} \left(  \| \vx_1 - \vx^\star\|^2 + \sum_{t \in \mathcal{I}_0} \tilde{\eta}_t^2 \right),
\end{align*}
where we use the following in the last inequality:
\begin{align*}
    \left\| \vx_{t+1} - \vx^\star \right\|^2
    \leq \| \vx_t - \vx^\star \|^2 - 2 \tilde{\eta}_t \left\langle \frac{\nabla f (\vx_t)}{\| \nabla f (\vx_t) \|}, \vx_t - \vx^\star \right\rangle + \tilde{\eta}_t^2.
\end{align*}
Since $\mathcal{I}_0 = \{1, \dots, \lfloor T/2 \rfloor \}$, a direct computation gives
\begin{align*}
    \sum_{t \in \mathcal{I}_0} \tilde{\eta}_t
    = \frac{\| \vx_1 - \vx^\star\|}{\sqrt{T}} \sum_{t=1}^{\lfloor T/2 \rfloor} \left( 1 - \frac{t}{T+1} \right)
    \geq \frac{3 (T - 1)}{8} \cdot \frac{\| \vx_1 - \vx^\star\|}{\sqrt{T}}
    \geq \| \vx_1 - \vx^\star\| \frac{\sqrt{T+1}}{4},
\end{align*}
where the last inequality holds for $T \geq 4$ (it is equivalent to $5 T^2 - 22 T + 9 \geq 0$), and \cref{eq:regime_of_sum_of_squared_stepsize,eq:regime_of_sum_of_stepsize} give
\begin{align*}
    \sum_{t \in \mathcal{I}_0} \tilde{\eta}^2_t \leq \left( \frac{T+1}{2} + 1 \right) \frac{\| \vx_1 - \vx^\star\|^2}{T} \leq 2 \| \vx_1 - \vx^\star\|^2.
\end{align*}
Since $\min_{t \in \mathcal{I}_0} \Phi(\vx_t) \sum_{t \in \mathcal{I}_0} \tilde{\eta}_t \leq \sum_{t \in \mathcal{I}_0} \tilde{\eta}_t \Phi(\vx_t) \leq 6 M_\nu^{\frac{1}{1+\nu}} \| \vx_1 - \vx^\star \|^2$, we conclude that there exists $s_0 \in \mathcal{I}_0$ such that
\begin{align}
\label{eq:initial_upper_bound_of_function_value}
    \Phi (\vx_{s_0}) \leq 24 M_\nu^\frac{1}{1 + \nu} \frac{\| \vx_1 - \vx^\star \|}{\sqrt{T}}.
\end{align}

Next, we discuss the case when $j \geq 1$.
Choose $s_j$ as follows for any $j$:
\begin{align*}
    s_j \coloneqq \argmin_{t \in \mathcal{I}_j} \Phi (\vx_t).
\end{align*}
From \cref{lemma:difference_of_function_values}, it holds for any $j \geq 1$
\begin{align*}
    \Phi (\vx_{s_j}) \leq \Phi (\vx_{s_{j-1}}) + M_\nu^\frac{1}{1 + \nu} \frac{\sum_{t \in \mathcal{I}_{j-1} \cup \mathcal{I}_{j}} \tilde{\eta}_t^2}{\sum_{t \in \mathcal{I}_{j}} \tilde{\eta}_t}.
\end{align*}
Using \cref{eq:regime_of_sum_of_stepsize,eq:regime_of_sum_of_squared_stepsize}, $T \geq 16$, and $2^{j} \leq 2^{J} < \frac{T+1}{\sqrt{T}}$, we obtain
\begin{align*}
    \sum_{t \in \mathcal{I}_{j-1} \cup \mathcal{I}_{j}} \tilde{\eta}_t^2
    &\leq \frac{\| \vx_1 - \vx^\star\|^2}{T} \left[ \left( \frac{T+1}{2^{j}} + 1 \right) \frac{1}{4^{j-1}} + \left( \frac{T+1}{2^{j+1}} + 1 \right) \frac{1}{4^{j}} \right] \\
    &= \frac{\| \vx_1 - \vx^\star\|^2}{8^{j}} \left[ \frac{9 (T+1)}{2 T} + \frac{5 \cdot 2^{j}}{T} \right]
    \leq \frac{\| \vx_1 - \vx^\star\|^2}{8^{j}} \cdot \frac{391}{64}
    \leq \| \vx_1 - \vx^\star\|^2 \frac{1}{2^{3j-4}}, \\
    \sum_{t \in \mathcal{I}_{j}} \tilde{\eta}_t
    &\geq \left( \frac{T+1}{2^{j+1}} - 1 \right) \frac{\| \vx_1 - \vx^\star\|}{2^{j+1} \sqrt{T}}
    \geq \frac{(T+1) \| \vx_1 - \vx^\star\|}{2^{2j+3} \sqrt{T}}
    \geq \| \vx_1 - \vx^\star\| \frac{\sqrt{T+1}}{2^{2j+3}},
\end{align*}
where, in the second line, we use $\frac{T+1}{2^{j+1}} - 1 \geq \frac{T+1}{2^{j+2}}$, which holds because $2^{j+2} \leq 2^{J+2} < \frac{4 (T+1)}{\sqrt{T}} \leq T + 1$ for $T \geq 16$. 
Therefore, we get
\begin{align*}
    \Phi (\vx_{s_j}) \leq \Phi (\vx_{s_{j-1}}) + M_\nu^\frac{1}{1 + \nu} \frac{\| \vx_1 - \vx^\star\|}{\sqrt{T+1}} \frac{1}{2^{j - 7}}.
\end{align*}
Unrolling this recursion, using \cref{eq:initial_upper_bound_of_function_value} and $\sum_{i=1}^{j} 2^{7 - i} \leq 2^{7}$, we obtain $\Phi(\vx_{s_j}) \leq (24 + 128) M_\nu^{\frac{1}{1+\nu}} \frac{\| \vx_1 - \vx^\star\|}{\sqrt{T}} = 152 M_\nu^{\frac{1}{1+\nu}} \frac{\| \vx_1 - \vx^\star\|}{\sqrt{T}}$ for all $j \in \{0, \dots, J\}$.
Raising both sides to the power $1 + \nu \leq 2$ gives the statement.
\end{proof}

\begin{lemma}
\label{lemma:upper_bound_of_e}
Suppose that \cref{assumption:holder_smoothness,assumption:convex} hold, $T \geq 16$, and $\tilde{\eta}_t = ( 1 - \tfrac{t}{T+1} ) \frac{\| \vx_1 - \vx^\star\|}{\sqrt{T}}$. We define
\begin{align*}
    \mathcal{E} (r) &\coloneqq \sum_{t=r}^T \tilde{\eta}_t \| \nabla f (\vx_t) \|, \\
    s_j &\coloneqq \argmin_{t \in \mathcal{I}_j} f (\vx_t).
\end{align*}
We have
\begin{align*}
    \begin{cases}
        \mathcal{E} (1) &\leq 2 M_\nu \| \vx_1 - \vx^\star \|^{1 + \nu} (T + 1)^{\frac{1 - \nu}{2}}, \\
        \mathcal{E} (s_{j-1}) &\leq M_\nu \left( \frac{\| \vx_1 - \vx^\star\|^2}{T} \right)^{\frac{1 + \nu}{2}} \left( 23104 + \frac{T + 1}{2^{(j-1) (2 + \nu)}} \right) \quad \text{for all $j \in \{ 1, \dots, J \}$}, \\
        \mathcal{E} (s_J) &\leq 23112 M_\nu \left( \frac{\| \vx_1 - \vx^\star\|^2}{T} \right)^{\frac{1 + \nu}{2}}.
    \end{cases}
\end{align*}
\end{lemma}
\begin{proof}
Note that $s_j$ is the choice made in the proof of \cref{lemma:upper_bound_of_phi}, so the bound of that lemma applies to $s_j$ for every $j \in \{0, \dots, J\}$. 
From \cref{lemma:holder_smooth} (using $\frac{M_\nu}{1+\nu} \leq M_\nu$), we have
\begin{align*}
    f (\vx_{t+1})
    &\leq f (\vx_t) + \langle \nabla f (\vx_t), \vx_{t+1} - \vx_t \rangle + M_\nu \tilde{\eta}_t^{1 + \nu}  \\
    &\leq f (\vx_t) - \tilde{\eta}_t \| \nabla f (\vx_t) \| + M_\nu \tilde{\eta}_t^{1 + \nu}.
\end{align*}
Summing up the above inequalities from $t=s$ to $T$ yields
\begin{align*}
    \mathcal{E} (s)
    = \sum_{t=s}^T \tilde{\eta}_t \| \nabla f (\vx_t) \|
    &\leq f (\vx_s) - f^\star + M_\nu \sum_{t=s}^T \tilde{\eta}_t^{1 + \nu}.
\end{align*}

We first consider the case when $s=1$.
\Cref{assumption:holder_smoothness,lemma:holder_smooth} imply that
\begin{align*}
    \mathcal{E} (1)
    &\leq M_\nu \| \vx_1 - \vx^\star \|^{1 + \nu} + M_\nu \sum_{t=1}^T \tilde{\eta}_t^{1 + \nu} \\
    &\leq M_\nu \| \vx_1 - \vx^\star \|^{1 + \nu} + M_\nu \| \vx_1 - \vx^\star \|^{1 + \nu} (T + 1)^{\frac{1 - \nu}{2}} \\
    &\leq 2 M_\nu \| \vx_1 - \vx^\star \|^{1 + \nu} (T + 1)^{\frac{1 - \nu}{2}}.
\end{align*}

Next, we consider the case when $j \in \{1, \dots, J \}$.
From \cref{lemma:upper_bound_of_phi}, we obtain
\begin{align*}
    \mathcal{E} (s_{j-1})
    &\leq  23104 M_\nu \left( \frac{\| \vx_1 - \vx^\star\|^2}{T} \right)^{\frac{1 + \nu}{2}} + M_\nu \sum_{t=s_{j-1}}^T \tilde{\eta}_t^{1 + \nu} \\
    &\leq  M_\nu \left( \frac{\| \vx_1 - \vx^\star\|^2}{T} \right)^{\frac{1 + \nu}{2}} \left( 23104 + \frac{T + 1}{2^{(j-1) (2 + \nu)}} \right),
\end{align*}
where we use that $\tilde{\eta}_t \leq 2^{-(j-1)} \frac{\| \vx_1 - \vx^\star\|}{\sqrt{T}}$ for all $t \geq s_{j-1}$ and that the number of such $t$ is $T + 1 - s_{j-1} \leq (T+1) 2^{-(j-1)}$, both by the definition of $s_{j-1} \in \mathcal{I}_{j-1}$.

Finally, we have
\begin{align*}
    \mathcal{E} (s_{J})
    &\leq  23104 M_\nu \left( \frac{\| \vx_1 - \vx^\star\|^2}{T} \right)^{\frac{1 + \nu}{2}} + M_\nu \sum_{t=s_J}^T \tilde{\eta}_t^{1 + \nu} \\
    &\leq  23112 M_\nu \left( \frac{\| \vx_1 - \vx^\star\|^2}{T} \right)^{\frac{1 + \nu}{2}},
\end{align*}
where we use $T - s_J + 1 \leq (T+1) 2^{-J} \leq 2 \sqrt{T}$ and $\tilde{\eta}_t \leq 2^{-J} \frac{\| \vx_1 - \vx^\star\|}{\sqrt{T}} \leq \frac{2 \| \vx_1 - \vx^\star\|}{T+1}$ for $t \geq s_J$, so that $\sum_{t=s_J}^T \tilde{\eta}_t^{1 + \nu} \leq 2 \sqrt{T} \bigl( \frac{2 \| \vx_1 - \vx^\star\|}{T + 1} \bigr)^{1+\nu} \leq 8 \bigl( \frac{\| \vx_1 - \vx^\star\|^2}{T} \bigr)^{\frac{1+\nu}{2}}$, in the last inequality.
\end{proof}

\begin{lemma}
\label{lemma:distance}
Suppose that \cref{assumption:convex,assumption:holder_smoothness} hold.
Normalized gradient descent with $\tilde{\eta}_t = (1 - \tfrac{t}{T+1}) \tfrac{\| \vx_1 - \vx^\star\|}{\sqrt{T}}$ satisfies
\begin{align*}
    \left\| \vx_{t} - \vx^\star \right\|^2
    &\leq 2 \left\| \vx_1 - \vx^\star \right\|^2,
\end{align*}
for any $t \in \{1, \dots, T+1\}$.
\end{lemma}
\begin{proof}
From the update rule of normalized gradient descent, we have
\begin{align*}
    \left\| \vx_{t+1} - \vx^\star \right\|^2
    &\leq \| \vx_t - \vx^\star \|^2 - 2 \tilde{\eta}_t \left\langle \frac{\nabla f (\vx_t)}{\| \nabla f (\vx_t) \|}, \vx_t - \vx^\star \right\rangle + \tilde{\eta}_t^2 \\
    &\leq \| \vx_t - \vx^\star \|^2 + \tilde{\eta}_t^2,
\end{align*}
where we use \cref{assumption:convex} in the second inequality.
Recursively applying the above inequality, we obtain
\begin{align*}
    \left\| \vx_{t+1} - \vx^\star \right\|^2
    &\leq \left\| \vx_1 - \vx^\star \right\|^2 + \sum_{s=1}^t \tilde{\eta}^2_s \\
    &\leq 2 \left\| \vx_1 - \vx^\star \right\|^2,
\end{align*}
where the last inequality uses $\tilde{\eta}_s \leq \frac{\| \vx_1 - \vx^\star\|}{\sqrt{T}}$ and $t \leq T$.
\end{proof}

\begin{proof}[Proof of \cref{theorem:convergenec_rate_of_last_with_adaptive_stepsize}]
We show that
\begin{align*}
    f (\vx_T) - f (\vx^\star) \leq 277384 \, M_\nu \bigl( \frac{\| \vx_1 - \vx^\star\|^2}{T} \bigr)^{\frac{1 + \nu}{2}},
\end{align*}
for every $T \geq 1$; we have not attempted to optimize this constant.
If $T < 16$, then \cref{lemma:holder_smooth} with $\vx = \vx^\star$ and \cref{lemma:distance} give $f (\vx_T) - f (\vx^\star) \leq \frac{M_\nu}{1+\nu} \| \vx_T - \vx^\star \|^{1+\nu} \leq M_\nu \bigl( 2 \| \vx_1 - \vx^\star \|^2 \bigr)^{\frac{1+\nu}{2}} \leq 32 M_\nu \bigl( \frac{\| \vx_1 - \vx^\star\|^2}{T} \bigr)^{\frac{1 + \nu}{2}}$, where we use $T < 16$ in the last step.
We now consider the case $T \geq 16$.

Let $\{ v_t \}_{t=0}^T$ be defined as follows:
\begin{align*}
    v_T &\coloneqq \frac{1}{\eta_T}, \\
    v_t &\coloneqq \frac{1}{\sum_{s=t+1}^T \eta_s}, \quad (\text{for all $t = 0, 1, \dots, T-1$}).
\end{align*}
As in the proof of \cref{lemma:average_rate_with_constant}, $\{ v_t \}_{t=0}^T$ is positive and nondecreasing,
\begin{align*}
    \eta_T v_{T - 1} = \eta_T v_T &= 1,
\end{align*}
and for any $t=1, 2, \dots, T-1$, we have
\begin{align*}
    \eta_t v_t - \left( \sum_{s=t}^T \eta_s \right) (v_t - v_{t-1})
    &=  - \left( \sum_{s=t+1}^T \eta_s \right) v_t + \left( \sum_{s=t}^T \eta_s \right) v_{t-1}
    = 0.
\end{align*}
Thus, using Lemma~\ref{lemma:zamani}, we obtain
\begin{align*}
    2 \left( f (\vx_T) - f (\vx^\star) \right)
    &\leq v_{0} \| \vx_1 - \vx^\star \|^2 + \sum_{t=1}^T v_t \tilde{\eta}_t^2.
\end{align*}
From the Cauchy–Schwarz inequality, it holds that for any $t \in \{0, 1, \dots, T-1\}$,
\begin{align}
\label{eq:application_of_cauchy_schwarz_inequality2}
    \left( \sum_{s=t+1}^T \frac{\tilde{\eta}_s}{\| \nabla f (\vx_s) \|} \right) \left( \sum_{s=t+1}^T \tilde{\eta}_s \| \nabla f (\vx_s) \| \right)
    \geq \left( \sum_{s=t+1}^T \tilde{\eta}_s \right)^2
    = \tilde{\eta}_t^2 \frac{(T - t)^2}{4}.
\end{align}
This implies that
\begin{align*}
    v_t \leq \frac{4 \sum_{s=t+1}^T \tilde{\eta}_s \| \nabla f (\vx_s) \|}{\tilde{\eta}_t^2 (T - t)^2}.
\end{align*}
In particular, we get
\begin{align*}
    v_0 \leq \frac{4 \sum_{s=1}^T \tilde{\eta}_s \| \nabla f (\vx_s) \|}{\tilde{\eta}_0^2 T^2}
    = \frac{4 \mathcal{E}(1)}{\| \vx_1 - \vx^\star\|^2 T}
    &\leq 16 M_\nu \| \vx_1 - \vx^\star\|^{-1 + \nu} \left( \frac{1}{T} \right)^{\frac{1 + \nu}{2}},
\end{align*}
where we use \cref{lemma:upper_bound_of_e} and $(T+1)^{\frac{1-\nu}{2}} \leq 2 T^{\frac{1-\nu}{2}}$.
Since $\mathcal{I}_0, \dots, \mathcal{I}_J$ partition $\{1, \dots, \max \mathcal{I}_J\}$ and $\max \mathcal{I}_J \leq T - 2$, we can split the sum over $t$ as follows:
\begin{align*}
    &2 \left( f (\vx_T) - f (\vx^\star) \right) \\
    &\leq 16 M_\nu \left( \frac{\| \vx_1 - \vx^\star\|^2}{T} \right)^{\frac{1 + \nu}{2}} + v_T \tilde{\eta}_T^2 + \sum_{t=1}^{T - 1} \frac{4}{(T - t)^2} \sum_{s=t+1}^T \tilde{\eta}_s \| \nabla f (\vx_s) \| \\
    &= 16 M_\nu \left( \frac{\| \vx_1 - \vx^\star\|^2}{T} \right)^{\frac{1 + \nu}{2}} + v_T \tilde{\eta}_T^2 \\
    &\quad + \sum_{j=0}^J \sum_{t \in \mathcal{I}_j \land t < T} \frac{4}{(T - t)^2} \sum_{s=t+1}^T \tilde{\eta}_s \| \nabla f (\vx_s) \|
    + \sum_{t = 1 + \max \mathcal{I}_J}^{T-1} \frac{4}{(T - t)^2} \sum_{s = t+1}^T \tilde{\eta}_s \| \nabla f (\vx_s) \|.
\end{align*}

Moreover, we have, for every $j \in \{0, \dots, J\}$,
\begin{align}
\label{eq:upper_bound_of_sum}
    \sum_{t \in \mathcal{I}_j \land t < T} \frac{1}{(T - t)^2}
    \leq \frac{2^{j+3}}{T + 1}.
\end{align}
Indeed, for $t \in \mathcal{I}_j$ we have $T + 1 - t > (T+1) 2^{-(j+1)}$ and hence, since $T + 1 - t$ is an integer, $T - t \geq \lfloor (T+1) 2^{-(j+1)} \rfloor \geq (T+1) 2^{-(j+2)}$, where the last inequality uses $(T+1) 2^{-(j+1)} \geq (T+1) 2^{-(J+1)} \geq 2$; moreover, $\sum_{1 \geq u \geq m} \frac{1}{u^2} \leq \frac{2}{m}$ for any integer $m \geq 1$.

Using \cref{lemma:upper_bound_of_e}, $\sum_{s=t+1}^{T} \tilde{\eta}_s \| \nabla f(\vx_s) \| = \mathcal{E}(t+1) \leq \mathcal{E}(1)$, and \cref{eq:upper_bound_of_sum}, we get
\begin{align*}
    \sum_{t \in \mathcal{I}_0} \frac{4}{(T - t)^2} \sum_{s=t+1}^T \tilde{\eta}_s \| \nabla f (\vx_s) \|
    &\leq 8 M_\nu \| \vx_1 - \vx^\star\|^{1 + \nu} (T + 1)^{\frac{1 - \nu}{2}} \sum_{t \in \mathcal{I}_0} \frac{1}{(T - t)^2} \\
    &\leq 64 M_\nu \left( \frac{\| \vx_1 - \vx^\star\|^2}{T + 1} \right)^{\frac{1 + \nu}{2}}
    \leq 64 M_\nu \left( \frac{\| \vx_1 - \vx^\star\|^2}{T} \right)^{\frac{1 + \nu}{2}}.
\end{align*}
Similarly, since $s_{j-1} \in \mathcal{I}_{j-1}$ implies $s_{j-1} \leq t$ for every $t \in \mathcal{I}_j$, we have $\mathcal{E}(t+1) \leq \mathcal{E}(s_{j-1})$ for $t \in \mathcal{I}_j$, and hence
\begin{align*}
    &\sum_{j=1}^{J} \sum_{t \in \mathcal{I}_j \land t < T} \frac{4}{(T - t)^2} \sum_{s=t+1}^T \tilde{\eta}_s \| \nabla f (\vx_s) \| \\
    &\leq \sum_{j=1}^{J} \sum_{t \in \mathcal{I}_j \land t < T} \frac{4}{(T - t)^2} \sum_{s=s_{j-1}}^T \tilde{\eta}_s \| \nabla f (\vx_s) \| \\
    &\leq M_\nu \left( \frac{\| \vx_1 - \vx^\star\|^2}{T} \right)^{\frac{1 + \nu}{2}} \sum_{j=1}^{J} \sum_{t \in \mathcal{I}_j \land t < T} \frac{4}{(T - t)^2}  \left( 23104 + \frac{T + 1}{2^{(j-1) (2 + \nu)}} \right) \\
    &\leq M_\nu \left( \frac{\| \vx_1 - \vx^\star\|^2}{T} \right)^{\frac{1 + \nu}{2}} \sum_{j=1}^{J}  \left( 92416 \frac{ 2^{j+3}}{T+1} + \frac{64}{2^{(j-1) (1 + \nu)}} \right) \\
    &\leq 369792 M_\nu \left( \frac{\| \vx_1 - \vx^\star\|^2}{T} \right)^{\frac{1 + \nu}{2}},
\end{align*}
where we use \cref{lemma:upper_bound_of_e} and \cref{eq:upper_bound_of_sum} for the second and third inequalities, and $\sum_{j=1}^{J} 2^{j+3} \leq 2^{J+4} < \frac{16 (T+1)}{\sqrt{T}} \leq 4 (T+1)$ together with $\sum_{j=1}^{J} 2^{-(j-1)(1+\nu)} \leq 2$ for the last inequality.
For the remaining iterations $t > \max \mathcal{I}_J \geq s_J$, we similarly have
\begin{align*}
    \sum_{t = 1 + \max \mathcal{I}_J}^{T-1} \frac{4}{(T - t)^2} \sum_{s = t+1}^T \tilde{\eta}_s \| \nabla f (\vx_s) \|
    &\leq \sum_{t = 1 + \max \mathcal{I}_J}^{T-1} \frac{4}{(T - t)^2} \sum_{s = s_J}^T \tilde{\eta}_s \| \nabla f (\vx_s) \| \\
    &\leq 7 \sum_{s = s_J}^T \tilde{\eta}_s \| \nabla f (\vx_s) \| \\
    &\leq 161784 M_\nu \left( \frac{\| \vx_1 - \vx^\star\|^2}{T} \right)^{\frac{1 + \nu}{2}},
\end{align*}
where we use $4 \sum_{u=1}^{\infty} \frac{1}{u^2} = \frac{2 \pi^2}{3} \leq 7$ in the second inequality and \cref{lemma:upper_bound_of_e} in the last inequality.
Finally, we also have
\begin{align*}
    v_T \tilde{\eta}_T^2 = \tilde{\eta}_T \| \nabla f (\vx_T) \| \leq \sum_{t = s_J}^T \tilde{\eta}_t \| \nabla f (\vx_t) \|
    \leq 23112 M_\nu \left( \frac{\| \vx_1 - \vx^\star\|^2}{T} \right)^{\frac{1 + \nu}{2}},
\end{align*}
where we use \cref{lemma:upper_bound_of_e} in the last inequality.
Combining the above inequalities, we obtain
\begin{align*}
    2 \left( f (\vx_T) - f (\vx^\star) \right)
    &\leq \left( 16 + 64 + 369792 + 161784 + 23112 \right) M_\nu \left( \frac{\| \vx_1 - \vx^\star\|^2}{T} \right)^{\frac{1 + \nu}{2}} \\
    &= 554768 \, M_\nu \left( \frac{\| \vx_1 - \vx^\star\|^2}{T} \right)^{\frac{1 + \nu}{2}},
\end{align*}
which is the desired result.
\end{proof}

\end{document}